\documentclass{article}  
\usepackage[final]{graphicx}   
\usepackage{psfrag}   
\usepackage{amsmath,amsfonts,amssymb,amsxtra,subeqnarray}

\newcommand {\Real}{\ensuremath{{\mathbb{R}}}}
\newcommand {\Natural}{\ensuremath{{\mathbb{N}}}}

\newcommand{\Q}{\ensuremath{\mathcal Q}}

\newcommand{\setS}{\ensuremath{\mathbb S}}

\newcommand{\setE}{\ensuremath{\mathcal E}}

\newcommand{\G}{\ensuremath{\mathcal G}}
\newcommand{\N}{\ensuremath{\mathcal N}}

\newcommand{\ex}{\ensuremath{{\mathbf{x}}}}

\newcommand{\one}{\ensuremath{{\mathbf{1}}}}

\newtheorem{theorem}{Theorem}
\newtheorem{corollary}{Corollary}

\newtheorem{lemma}{Lemma}

\newtheorem{fact}{Fact}
\newtheorem{definition}{Definition}
\newtheorem{remark}{Remark}
\newtheorem{assumption}{Assumption}

\newenvironment{proof}{\noindent {\bf Proof.}}{\hfill \hspace*{1pt}\hfill$\blacksquare$}

\begin{document}
\title{Sufficient conditions on observability grammian 
for synchronization in arrays of coupled time-varying linear
systems}
\author{S. Emre Tuna\\
{\small {\tt tuna@eee.metu.edu.tr}} }
\maketitle

\begin{abstract}
Synchronizability of stable, output-coupled, identical, time-varying
linear systems is studied. It is shown that if the observability
grammian satisfies a persistence of excitation condition, then there
exists a bounded, time-varying linear feedback law that yields
exponential synchronization for all fixed, asymmetrical
interconnections with connected graphs. Also, a weaker condition on
the grammian is given for asymptotic synchronization. No assumption is
made on the strength of coupling. Moreover, related to the main
problem, a particular array of output-coupled systems that is
pertinent to much-studied consensus problems is investigated. In this
array, the individual systems are integrators with identical,
time-varying, symmetric positive semi-definite output
matrices. Trajectories of this array are shown to stay bounded using a
time-invariant, quadratic Lyapunov function. Also, sufficient
conditions on output matrix for synchronization are provided. All of
the results in the paper are generated for both continuous time and
discrete time.
\end{abstract}

\section{Introduction}
When do the trajectories of a number of coupled individual systems
converge to each other?  This question outlines the multifaceted
problem of {\em synchronization stability}. Unknotting this problem
requires understanding the interplay of two pieces: the set of
individual systems' dynamics and the (varying) topology of their
coupling, i.e. who influences whom and how strongly. The general
problem is insuperably difficult, which has led people to a number of
simplifications, justifiable for certain applications. For instance,
when the individual system dynamics is taken to be an integrator, by
using convexity arguments, trajectories have been shown to converge to
a fixed point in space as long as the (directed, time-varying)
interconnection satisfies a fairly weak connectedness
condition. Since, once synchronized, the righthand sides of the
systems vanish, the word {\em consensus} is used when referring to
this case; see, for instance,
\cite{moreau05,ren05,angeli06,lin07}. Another direction of
investigation is fueled by the fact that the speed/occurrence of
synchronization is related to the coupling (strength) between the
individual systems. Studies concentrated on understanding this
relation have been fruitful and significant results have emerged. We
now know that the spectrum of the interconnection matrix is where we
have to look at if we want to measure the strength of coupling in
order to determine whether synchronization will take place or
not. Roughly speaking, under the assumption that some Lyapunov
function (related to the individual system dynamics only) exists, one
can guarantee stability of synchronization if the coupling strength is
larger than some threshold; see, for instance,
\cite{wu95,pecora98,wu05,belykh06}.  There are numerous other
interesting research directions accommodating notable works in
synchronization stability. We refer the interested reader to the
surveys \cite{strogatz01,wang02}, \cite[Sec. 5]{boccaletti06}.

A fundamental case in synchronization stability concerns with
output-coupled identical linear systems under fixed
interconnection. The problem is considered in \cite{tuna08} for
time-invariant discrete-time systems and in \cite{trac08a} for
continuous-time systems (as a generalization of Luenberger observer)
leading to the following result: ``If an individual system is
detectable from its output and its system matrix is neutrally stable,
then there exists a linear feedback law under which the trajectories
of the coupled replicas of the individual system exponentially
synchronize provided that the (directed) graph representing the
interconnection is connected.'' We emphasize that (i) the result needs
no assumption on the strength of coupling and (ii) synchronizing
feedback law is independent of the number of systems and their
interconnection. In this paper we extend this result for time-varying
linear systems.

For a time-varying pair $(C,\,A)$, where $A(\cdot)$ is the system
matrix and $C(\cdot)$ is the output matrix, we first define {\em
synchronizability} (with respect to set of all connected
interconnections.)  Roughly, a pair $(C,\,A)$ is synchronizable if one
can find a bounded time-varying linear feedback law $L(\cdot)$ under
which the trajectories of the coupled replicas of the individual
system described by triple $(C,\,A,\,L)$ synchronize for all connected
interconnections. Then we study the conditions that would imply
synchronizability. The assumptions and results almost parallel the
time-invariant case. The assumption we make on the system matrix is
that its state transition matrix is bounded in both forward and
backward time\footnote{Bounded state transition matrix assumption can
be encountered in seemingly different problems in the literature
whenever observability is at stake; see, for instance,
\cite{chen97}.}, which yields (considering trajectories) sustained and
bounded oscillations. Boundedness in forward time is necessary for
stability because we make no assumption on the strength of
coupling. Boundedness in backward time can be relaxed at the expense
of complicacy of analysis and need for additional technical
assumptions on pair $(C,\,A)$. For simplicity, therefore, we choose to
keep it.  One of the contributions of this work are in establishing
the following results:
\begin{itemize}
\item 
If pair $(C,\,A)$ is asymptotically observable then it is
synchronizable.
\item  If pair $(C,\,A)$ is uniformly observable then it is exponentially 
synchronizable.\footnote{Along with these results, we also 
provide a synchronizing feedback law $L(\cdot)$ in the paper.}
\end{itemize}
{\em Asymptotic observability} we define as that the integrand of the
observability grammian satisfies a general (yet technical)
condition. This condition, which we name {\em sufficiency of excitation
\footnote{See Definition~\ref{def:suff}.}},
is significantly weaker than persistence of excitation and allows the
following result, cf.~\cite[Thm.~2.5.1]{sastry89}.
\begin{itemize}
\item Let $Q$ be bounded and $Q(t)=Q(t)^{T}\geq 0$ for all $t\geq 0$. 
Linear system $\dot{x}=-Q(t)x$ satisfies $\lim_{t\to\infty}x(t)=0$ if
$Q$ is sufficiently exciting.
\end{itemize}
Uniform observability, on the other
hand, is quite a standard concept, which is more or less equivalent to
that the integrand (summand) of the observability grammian is
persistently exciting.

To obtain the above listed results we first study synchronization
stability of a particular type of array. This array is pertinent to
consensus problem (for trajectories are static once synchronized) yet
different from the usual array of interest in consensus problems
\cite{blondel05}.  Our second contribution in this paper is 
in analyzing this new type of consensus array and, consequently,
unraveling two arrays' similarities and differences. In addition, we
investigate the stability of their union. The array dynamics generally
studied in consensus problems is
\begin{eqnarray}\label{eqn:intro1}
\dot{x}_{i}=\sum_{j=1}^{p}\gamma_{ij}(t)(x_{j}-x_{i})
\end{eqnarray}
where $x_{i}\in\Real^{n}$ is the state of the $i$th system and
$\gamma_{ij}(t)\geq 0$ for all $t$. What is known about this array is
that its trajectories are bounded. In fact, the convex hull of the
states ${\rm co}\{x_{1},\,\ldots,\,x_{p}\}$ is forward invariant
regardless of the evolution of $\gamma_{ij}(\cdot)$. Moreover, if
certain connectedness property is satisfied by the graph described by
$\{\gamma_{ij}\}$, then trajectories
$x_{i}(\cdot)$ meet at some common point, i.e.
reach consensus. Finally, in general, there does not exist a
quadratic Lyapunov function to establish stability;
so the convex hull of the states is used instead
\cite{moreau05}. The array considered in this paper is
\begin{eqnarray}\label{eqn:intro2}
\dot{x}_{i}=\sum_{j=1}^{p}\gamma_{ij}(y_{j}-y_{i})\,,\quad y_{i}=Q(t)x_{i}
\end{eqnarray}
where time-varying output matrix $Q(\cdot)$ is symmetric positive
semi-definite and $\gamma_{ij}$ is fixed. Below we list our findings
residing in Section~\ref{sec:spsd}.
\begin{itemize}
\item Like array~\eqref{eqn:intro1}, 
trajectories of array~\eqref{eqn:intro2} are bounded.
\item Unlike array~\eqref{eqn:intro1}, there exists a quadratic
Lyapunov function\footnote{However,
the convex hull is no longer forward invariant.} for array~\eqref{eqn:intro2}.
\item For $Q$ sufficiently exciting, trajectories 
of array~\eqref{eqn:intro2} reach consensus for all connected
interconnections. The point of consensus is independent of the
evolution of $Q$.
\end{itemize}
We also look at the union of the two cases
$\dot{x}_{i}=\sum_{j=1}^{p}\gamma_{ij}(t)(y_{j}-y_{i})$,
$y_{i}=Q(t)x_{i}$. We find that unbounded trajectories may result from
this situation, hence stability is no longer guaranteed.

The outline of the paper is as follows. After introducing notation and
basic definitions, we define synchronizability and give the formal
problem statement for continuous-time linear time-varying systems in
Section~\ref{sec:ps}. This section also is where we draw the simple
link between synchronization of time-varying linear systems and
consensus of array~\eqref{eqn:intro2}. In Section~\ref{sec:spsd} we
establish the stability of array~\eqref{eqn:intro2} via a quadratic
Lyapunov function and construct (observability) conditions on $Q$
yielding consensus. In Section~\ref{sec:main} we interpret these
conditions through the observability grammian of time-varying pair
$(C,\,A)$ and establish our main results. Finally, in
Section~\ref{sec:dt}, we generate the discrete-time versions of the
continuous-time results.

\section{Notation and definitions}
Let $\Natural$ denote the set of nonnegative integers and $\Real_{\geq
0}$ the set of nonnegative real numbers. The meaning of
$\Natural_{\geq k}$ is the obvious. Let $|\cdot|$ denote (induced)
2-norm.  Identity matrix in $\Real^{n\times n}$ is denoted by $I_{n}$.
The set of all symmetric positive semi-definite (SPSD) matrices in
$\Real^{n\times n}$ is denoted by $\Q_{n}$. We also define
$\overline{\Q}_{n}:=\{R\in\Q_{n}:|R|\leq 1\}$. Let $\one\in\Real^{p}$
denote the vector with all entries equal to one. The smallest and
largest singular values of $A\in\Real^{m\times n}$ are, respectively,
denoted by $\sigma_{\rm min}(A)$ and $\sigma_{\rm max}(A)$. {\em
Kronecker product} of $A\in\Real^{m\times n}$ and $B\in\Real^{p\times
q}$ is
\begin{eqnarray*}
A\otimes B:=
\left[\!\!
\begin{array}{ccc}
a_{11}B & \!\cdots\! & a_{1n}B\\
\vdots  &\! \ddots\! & \vdots\\
a_{m1}B &\! \cdots\! & a_{mn}B
\end{array}\!\!
\right]
\end{eqnarray*}
Kronecker product comes with the following properties: $(A\otimes B)(C\otimes
D)=(AC)\otimes(BD)$ (provided that products $AC$ and $BD$ are allowed);
$A\otimes B+A\otimes C=A\otimes(B+C)$ (for $B$ and $C$ that are of the
same size); and $(A\otimes B)^{T}=A^{T}\otimes B^{T}$. Moreover, the singular
values of $(A\otimes B)$ equal the (pairwise) product of singular values of
$A$ and $B$.

A ({\em directed}) {\em graph} is a pair $(\N,\,\setE)$ where $\N$ is
a nonempty finite set (of {\em nodes}) and $\setE$ is a finite
collection of ordered pairs ({\em edges}) $(n_{i},\,n_{j})$ with
$n_{i},\,n_{j}\in\N$. A {\em directed path} from $n_{1}$ to $n_{\ell}$
is a sequence of nodes $(n_{1},\,n_{2},\,\ldots,\,n_{\ell})$ such that
$(n_{i},\,n_{i+1})$ is an edge for
$i\in\{1,\,2,\,\ldots,\,\ell-1\}$. A graph is {\em connected} if it
has a node to which there exists a directed path from every other
node.\footnote{Note that this definition of connectedness for directed
graphs is weaker than strong connectivity and stronger than weak
connectivity.}
The graph of a matrix $M:=[m_{ij}]\in\Real^{p\times p}$ is the pair
$(\N,\,\setE)$, where $\N =\{n_{1},\,n_{2},\,\ldots,\,n_{p}\}$ and
$\setE$ is such that $(n_{i},\,n_{j})\in\setE$ iff $m_{ij}>0$. Matrix
$M$ is said to be {\em connected} when its graph is connected.

Throughout the paper $\Gamma:=[\gamma_{ij}]\in\Real^{p\times p}$ will
represent an {\em interconnection (in the continuous-time sense)}
satisfying $\gamma_{ij}\geq 0$ for $i\neq j$ and
$\gamma_{ii}=-\sum_{j\neq i}\gamma_{ij}$ for all $i$. It immediately
follows that $\lambda=0$ is an eigenvalue with eigenvector $\one$, that is, 
$\Gamma\one=0$. For $\Gamma$ connected, eigenvalue $\lambda=0$ is distinct and all the other
eigenvalues have real parts strictly negative.  Let $r\in\Real^{p}$ satisfy
\begin{subeqnarray}\label{eqn:r}
r^{T}\Gamma &=& 0\\
r^{T}\one &=& 1\,.
\end{subeqnarray}
Then $r$ is unique (for $\Gamma$ connected) and satisfies
$\lim_{t\to\infty}e^{\Gamma{t}}={\one}r^{T}$. Also, $r$ has no
negative entry. We denote by $\G_{>0}$ the set of all connected
interconnections, i.e. $\G_{>0}=\{\Gamma\in\Real^{p\times p}:
\Gamma\ \mbox{connected interconnection},\ p=2,\,3,\,\ldots\}$.

Matrix $\Lambda:=[\lambda_{ij}]\in\Real^{p\times p}$ denotes an {\em
interconnection (in the discrete-time sense)} satisfying
$\lambda_{ij}\geq 0$ for all $i,\,j$ and $\sum_{j}\lambda_{ij}=1$ for
all $i$. It follows that $\lambda=1$ is an eigenvalue with eigenvector
$\one$, that is, $\Lambda\one=\one$. For a connected $\Lambda$,
eigenvalue $\lambda=1$ is distinct and all the other eigenvalues lie
strictly within the unit circle.  Let $r\in\Real^{p}$ satisfy
\begin{subeqnarray}\label{eqn:rdt}
r^{T}\Lambda &=& r^{T}\\
r^{T}\one &=& 1\,.
\end{subeqnarray}
Then $r$ is unique (for $\Lambda$ connected) and satisfies
$\lim_{k\to\infty}\Lambda^{k}={\one}r^{T}$. Also, $r$ has no negative
entry. By slight abuse of notation (yet with a negligible risk of
ambiguity) we will let $\G_{>0}$ also denote the set of all
(discrete-time) connected interconnections $\Lambda$.

Let $\setS\in\{\Real_{\geq 0},\,\Natural\}$. Given maps
$\xi_{i}:\setS\to\Real^{n}$ for $i=1,\,2,\,\ldots,\,p$ and a map
$\bar\xi:\setS\to\Real^{n}$, the elements of the set
$\{\xi_{i}(\cdot):i=1,\,2,\,\ldots,\,p\}$ are said to {\em synchronize
to} $\bar{\xi}(\cdot)$ if $|\xi_{i}(s)-\bar\xi(s)|\to 0$ as
$s\to\infty$ for all $i$. They are said to {\em synchronize} if they
synchronize to some $\bar{\xi}(\cdot)$. Moreover, if there exists a
pair of positive real numbers $(c,\,\alpha)$ such that
$\max_{i}|\xi_{i}(s)-\bar\xi(s)|\leq ce^{-\alpha s}$ for all $s$, then
$\xi_{i}(\cdot)$ are said to {\em exponentially} synchronize.

\section{Problem statement}\label{sec:ps}
For a given interconnection $\Gamma=[\gamma_{ij}]\in\Real^{p\times p}$, let an array of
$p$ linear systems be
\begin{subeqnarray}\label{eqn:array}
{\dot{x}}_{i} &=& A(t)x_{i}+u_{i} \label{eqn:a}\\
 y_{i} &=& C(t)x_{i}\label{eqn:b}\\
 z_{i} &=& \displaystyle \sum_{j\neq i}\gamma_{ij}(y_{j}-y_{i})\label{eqn:c}
\end{subeqnarray}
where $x_{i}\in\Real^{n}$ is the {\em state}, $u_{i}\in\Real^{n}$ is
the {\em input}, $y_{i}\in\Real^{m}$ is the {\em output}, and
$z_{i}\in\Real^{m}$ is the {\em coupling} of the $i$th system for
$i=1,\,2,\,\ldots,\,p$. For each $t\in\Real$ we have
$A(t)\in\Real^{n\times n}$ and $C(t)\in\Real^{m\times n}$. The
solution of $i$th system at time $t\geq 0$ is denoted by $x_{i}(t)$.
We denote by $\Phi_{A}(\cdot,\,\cdot)$ 
the state transition matrix for $A$,
i.e. the unique solution of the matrix differential equation
\begin{eqnarray*}
\dot{\Phi}_{A}(t,\,t_{0})=A(t)\Phi_{A}(t,\,t_{0})
\end{eqnarray*} 
with $\Phi_{A}(t_{0},\,t_{0})=I_{n}$. Also, recall that the
observability grammian of pair $(C,\,A)$ is given by
\begin{eqnarray}\label{eqn:grammian}
W_{\rm o}(t_{0},\,t):=
\int_{t_{0}}^{t}\Phi_{A}^{T}(\tau,\,t_{0})C^{T}(\tau)C(\tau)
\Phi_{A}(\tau,\,t_{0})d\tau
\end{eqnarray}
for $t_{0},\,t\in\Real$. We will henceforth assume that the integrand
of the grammian is Riemann-integrable.

\begin{definition}[Synchronizability]\label{def:sync}
Given functions $A:\Real\to\Real^{n\times n}$ and
$C:\Real\to\Real^{m\times n}$; pair $(C,\,A)$ is said to be {\em
synchronizable (with respect to $\G_{>0}$)} if there exists a bounded,
time-varying linear feedback law $L:\Real\to\Real^{n\times m}$ such
that for each $\Gamma\in\G_{>0}$, solutions $x_{i}(\cdot)$ of
array~\eqref{eqn:array} with $u_{i}=L(t)z_{i}$ synchronize for all
initial conditions.
\end{definition}

Our objective in this paper is to find sufficient conditions on pair
$(C,\,A)$, in particular on the observability grammian
\eqref{eqn:grammian}, for synchronizability and to design a
synchronizing feedback law $L$ when proposed conditions are met.

The above statement of our objective almost suggests that we first
find sufficient conditions and search for an $L$ only {\em
afterwards}.  However we adopt the opposite approach. We choose first
to construct an $L$ and then work out the conditions on $(C,\,A)$
for synchronization under such feedback law. Given $(C,\,A)$ let
\begin{eqnarray}\label{eqn:L}
L(t):=\Phi_{A}(t,\,0)\Phi_{A}^{T}(t,\,0)C^{T}(t)\,.
\end{eqnarray}
For interconnection $\Gamma\in\Real^{p\times p}$ consider
array~\eqref{eqn:array} with $u_{i}=L(t)z_{i}$. We can write
\begin{eqnarray}\label{eqn:x}
\dot{x}_{i}=A(t)x_{i}+L(t)C(t)\sum_{j\neq i}\gamma_{ij}(x_{j}-x_{i})\,.
\end{eqnarray} 
Let us define the auxiliary variable $\xi_{i}\in\Real^{n}$ as
\begin{eqnarray}\label{eqn:xi}
\xi_{i}(t):=\Phi_{A}(0,\,t)x_{i}(t)
\end{eqnarray}
for $i=1,\,2,\,\ldots,\,p$ and $t\geq 0$.
Combining \eqref{eqn:L}, \eqref{eqn:x}, and \eqref{eqn:xi} we obtain
\begin{eqnarray}\label{eqn:aux}
\dot{\xi}_{i}=\Phi_{A}^{T}(t,\,0)C^{T}(t)C(t)\Phi_{A}(t,\,0)\sum_{j\neq
i}\gamma_{ij}(\xi_{j}-\xi_{i})\,.
\end{eqnarray}
Now note that if $\Phi_{A}$ is bounded, then synchronization of
solutions $\xi_{i}(\cdot)$ implies synchronization of solutions
$x_{i}(\cdot)$ by \eqref{eqn:xi}.  Moreover, if $C$ is bounded as
well, then boundedness of $L$ is guaranteed by \eqref{eqn:L}. Based on
this simple observation let us write the following assumption to be
invoked later.
\begin{assumption}[Boundedness]\label{assume:AC}
For $A:\Real\to\Real^{n\times n}$ and $C:\Real\to\Real^{m\times n}$
following hold.
\begin{itemize}
\item[{(a)}] There exists $\bar{a}\geq 1$ such that 
$|\Phi_{A}(t_{1},\,t_{2})|\leq \bar{a}$ for all $t_{1},\,t_{2}\geq 0$.
\item[{(b)}] There exists $\bar{c}\geq 1$ such that $|C(t)|\leq \bar{c}$ 
for all $t\geq 0$.
\end{itemize}
\end{assumption}

\begin{remark}
Note that in the time-invariant case Assumption~\ref{assume:AC}(b)
comes for free and Assumption~\ref{assume:AC}(a) boils down to that
matrix $A$ is neutrally stable (in the continuous-time sense) with all
its eigenvalues residing on the imaginary axis.
\end{remark}

The second point we want to make is that the term multiplying the sum
in \eqref{eqn:aux} is the integrand of the observability grammian,
which is SPSD at each $t$. We elaborate on this
fact in the next section.

\section{Synchronization under SPSD matrix}\label{sec:spsd}

For a given interconnection $\Gamma=[\gamma_{ij}]\in\Real^{p\times
p}$, let an array of $p$ systems be
\begin{eqnarray}\label{eqn:system}
\dot{x}_{i}=Q_{t}\sum_{j\neq i}\gamma_{ij}(x_{j}-x_{i})
\end{eqnarray}
where $x_{i}\in\Real^{n}$ is the state of the $i$th system (for
$i=1,\,2,\,\ldots,\,p$) and $Q_{t}\in\Real^{n\times n}$ is SPSD for
each $t\geq 0$. We assume $Q:\Real_{\geq 0}\to\Q_{n}$ to be
Riemann-integrable. By letting
\begin{eqnarray*}
\ex:=\left[\!\!\begin{array}{c}
x_{1}\\\vdots\\x_{p}
\end{array}\!\!\right]
\end{eqnarray*}
we can rewrite \eqref{eqn:system} more compactly as
\begin{eqnarray}\label{eqn:systemkron}
\dot{\ex}=(\Gamma\otimes Q_{t})\ex\,.
\end{eqnarray}

\begin{remark}\label{rem:bnd}
Sometimes we need function $Q:\Real_{\geq 0}\to \Q_{n}$ be
bounded on the interval $[0,\,\infty)$, i.e. there exists
$h\geq 1$ such that $|Q_{t}|\leq h$ for all $t$. Note that
\eqref{eqn:systemkron} can be written as
\begin{eqnarray*}
\dot{\ex}=\left(h\Gamma\otimes \frac{Q_{t}}{h}\right)\ex\,.
\end{eqnarray*}
Now, since $\Gamma$ is an interconnection, so is $h\Gamma$. Also,
connectedness is invariant under multiplication by a positive scalar,
i.e. $\Gamma$ is connected if and only if $h\Gamma$ is. Finally,
observe that $Q_{t}/h\in\overline{\Q}_{n}$. Without loss of generality
(for our purposes) therefore we can can
take $h$ to be unity, which lets us consider $Q:\Real_{\geq
0}\to\overline{\Q}_{n}$ whenever we need $Q$ be bounded.
\end{remark}

In the rest of this section we first show that the origin of
system~\eqref{eqn:systemkron} is stable regardless of interconnection
$\Gamma\in\Real^{p\times p}$ and function $Q:\Real\to\Q_{n}$. Then,
under connectedness of $\Gamma$, which is obviously necessary for
synchronization, we work out some sufficient conditions on function
$Q$ to establish synchronization of solutions $x_{i}(\cdot)$ of
array~\eqref{eqn:system}. Finally, we provide two theorems to make the
picture that we want to give in this section closer to complete. One
of those theorems states that time-invariance of interconnection
$\Gamma$ in \eqref{eqn:systemkron} is necessary for stability. With
the other one, we aim to show that the sufficient conditions that we
will have proposed on $Q$ for synchronization cannot be {\em readily}
relaxed into a less technical one without sacrificing generality.

\subsection{Stability}

\begin{lemma}\label{lem:omega}
Let interconnection $\Gamma\in\Real^{p\times p}$ be connected and
$r\in\Real^{p}$ satisfy \eqref{eqn:r}. Then, there
exists symmetric positive definite matrix
$\Omega\in\Real^{p\times p}$ such that
\begin{eqnarray}\label{eqn:lyap}
(\Gamma-\one{r^{T}})^{T}\Omega+\Omega(\Gamma-\one{r^{T}})=-I_{p}\,.
\end{eqnarray}
\end{lemma}

\begin{proof}
Consider matrix $\Gamma-\one{r}^{T}$. Observe that
$(\Gamma-\one{r}^{T})^{k}=\Gamma^{k}+(-1)^{k}\one{r^{T}}$ for $k\in\Natural$. For
$t\in\Real$, therefore we can write
\begin{eqnarray*}
e^{(\Gamma-1r^{T})t}
&=&I_{p}+t(\Gamma-1r^{T})+\frac{t^{2}}{2}(\Gamma-1r^{T})^{2}+\ldots\\
&=&\left(I_{p}+t\Gamma+\frac{t^{2}}{2}\Gamma^{2}+\ldots\right)
-\left(t\one{r^{T}}-\frac{t^{2}}{2}\one{r}^{T}+\ldots\right)\\
&=&e^{\Gamma{t}}-(1-e^{-t})\one{r^{T}}\,.
\end{eqnarray*} 
Consequently, $\lim_{t\to\infty}e^{(\Gamma-\one{r^{T}})t}=0$; and we deduce
that $[\Gamma-\one{r^{T}}]$ is Hurwitz. Therefore Lyapunov
equation \eqref{eqn:lyap} admits a symmetric positive definite
solution $\Omega$.
\end{proof}

\begin{lemma}\label{lem:Vdot}
Let interconnection $\Gamma\in\Real^{p\times p}$ be connected,
$r\in\Real^{p}$ satisfy \eqref{eqn:r}, and symmetric positive definite
matrix $\Omega\in\Real^{p\times p}$ satisfy \eqref{eqn:lyap}. Define
$V:\Real^{np}\to\Real_{\geq 0}$ as $V(\ex):=\ex^{T}(\Omega\otimes
I_{n})\ex$.  Then, for all $Q:\Real_{\geq 0}\to\Q_{n}$ and all $t\geq
0$, solution of system~\eqref{eqn:systemkron} satisfies
\begin{eqnarray*}
\frac{d}{dt}V(\ex(t)-\bar{\ex})
=-(\ex(t)-\bar{\ex})^{T}(I_{p}\otimes Q_{t})(\ex(t)-\bar{\ex})
\end{eqnarray*}
where $\bar{\ex}:=(\one{r^{T}}\otimes I_{n})\ex(0)$.
\end{lemma}

\begin{proof}
Observe that $(\one r^{T}\otimes I_{n})\dot{\ex}(t)=0$, which implies
$(\one r^{T}\otimes I_{n}){\ex}(t)=\bar{\ex}$ for all $t\geq
0$. Whence $\dot{\ex}(t)=((\Gamma-\one r^{T})\otimes
Q_{t})(\ex(t)-\bar{\ex})$. We can therefore write
\begin{eqnarray*}
\frac{d}{dt}V(\ex(t)-\bar{\ex})
&=& \dot{\ex}(t)^{T}(\Omega\otimes I_{n})(\ex(t)-\bar{\ex})
+(\ex(t)-\bar{\ex})^{T}(\Omega\otimes I_{n})\dot{\ex}(t)\\
&=& (\ex(t)-\bar{\ex})^{T}((\Gamma-\one r^{T})\otimes Q_{t})^{T}(\Omega\otimes I_{n})
(\ex(t)-\bar{\ex})\\
&& +(\ex(t)-\bar{\ex})^{T}(\Omega\otimes I_{n})((\Gamma-\one r^{T})\otimes Q_{t})
(\ex(t)-\bar{\ex})\\
&=& (\ex(t)-\bar{\ex})^{T}\left[((\Gamma-\one r^{T})^{T}\Omega
+\Omega(\Gamma-\one r^{T}))\otimes Q_{t}\right](\ex(t)-\bar{\ex})\\
&=& -(\ex(t)-\bar{\ex})^{T}(I_{p}\otimes Q_{t})(\ex(t)-\bar{\ex})\,.
\end{eqnarray*}
Hence the result.
\end{proof}

\begin{theorem}[Stability]\label{thm:stability}
Given interconnection $\Gamma\in\Real^{p\times p}$, there exists
$\alpha>0$ such that, for all $Q:\Real_{\geq 0}\to\Q_{n}$, solution of
system~\eqref{eqn:systemkron} satisfies
\begin{eqnarray*}
|\ex(t)|\leq\alpha|\ex(0)|
\end{eqnarray*}
for all $t\geq 0$.
\end{theorem}

\begin{proof}
Interconnection $\Gamma$ is similar to a block diagonal matrix ${\rm
diag}(\Gamma_{1},\,\Gamma_{2},\,\ldots,\,\Gamma_{q})$ in
$\Real^{p\times p}$ such that $\Gamma_{i}\in\Real^{p_{i}\times p_{i}}$
for $i=1,\,2,\,\ldots,\,q$ is a connected interconnection if
$p_{i}\geq 2$ and $\Gamma_{i}=0$ otherwise. (Integer $q$ equals the
number of eigenvalues of $\Gamma$ at the origin.) Since ${\rm
diag}(\Gamma_{1},\,\Gamma_{2},\,\ldots,\,\Gamma_{q})\otimes Q_{t}=
{\rm diag}(\Gamma_{1}\otimes Q_{t},\,\Gamma_{2}\otimes
Q_{t},\,\ldots,\,\Gamma_{q}\otimes Q_{t})$ without loss of generality
it suffices to check two cases: (i) $\Gamma=0$; and (ii) $\Gamma$ is
connected. First case is trivial; so let us suppose $\Gamma$ is
connected.

Now, let $r\in\Real^{p}$ and symmetric positive definite matrix
$\Omega\in\Real^{p\times p}$ satisfy \eqref{eqn:r} and
\eqref{eqn:lyap}, respectively. Given $Q:\Real_{\geq 0}\to\Q_{n}$,
consider system~\eqref{eqn:systemkron}. Let $\bar{\ex}=(\one
r^{T}\otimes I_{n})\ex(0)$. Recalling that $r$ has no negative entry,
we can write
\begin{eqnarray*}
|\bar\ex|
&\leq&|\one r^{T}||\ex(0)|\\
&\leq&|\one||\ex(0)|\\
&=&\sqrt{n}|\ex(0)|\,.
\end{eqnarray*}
Let $V(\ex)=\ex^{T}(\Omega\otimes I_{n})\ex$. Lemma~\ref{lem:Vdot}
yields that $V(\ex(\cdot)-\bar\ex)$ is nonincreasing.  Hence
\begin{eqnarray*}
|\ex(t)|
&\leq& |\ex(t)-\bar\ex|+|\bar\ex|\\
&\leq& \frac{1}{\sqrt{\sigma_{\rm min}
(\Omega)}}\sqrt{V(\ex(t)-\bar\ex)}+|\bar\ex|\\
&\leq& \frac{1}{\sqrt{\sigma_{\rm min}
(\Omega)}}\sqrt{V(\ex(0)-\bar\ex)}+|\bar\ex|\\
&\leq& \sqrt{\frac{\sigma_{\rm max}(\Omega)}{\sigma_{\rm min}
(\Omega)}}|\ex(0)-\bar\ex|+|\bar\ex|\\
&\leq& \left(\sqrt{\frac{\sigma_{\rm max}(\Omega)}{\sigma_{\rm min}
(\Omega)}}(1+\sqrt{n})+\sqrt{n}\right)|\ex(0)|
\end{eqnarray*}
for all $t\geq 0$. 
\end{proof}

Theorem~\ref{thm:stability} establishes stability. That is, for a
fixed interconnection\footnote{Later in the section, we will also
investigate whether stability is preserved when both $\Gamma$ and $Q$
are time-varying.}  $\Gamma$, which need not be connected, solutions
$x_{i}(\cdot)$ of array~\eqref{eqn:system} stay within a bounded
region (that depends only on initial conditions $x_{i}(0)$ and
interconnection $\Gamma$) for all time-varying SPSD matrix
$Q$. Whenever $\Gamma$ is connected, that bounded region can be
described quite precisely. See the next result, which is a direct
consequence of Lemma~\ref{lem:Vdot}.

\begin{theorem}
Given connected interconnection $\Gamma\in\Real^{p\times p}$, 
let $r\in\Real^{p}$ and symmetric positive definite matrix
$\Omega\in\Real^{p\times p}$ satisfy \eqref{eqn:r} and
\eqref{eqn:lyap}, respectively. Then, 
for all $Q:\Real_{\geq 0}\to\Q_{n}$, solutions $x_{i}(\cdot)$ of
array~\eqref{eqn:system} satisfy, for all $t\geq 0$,
\begin{eqnarray*}
|x_{i}(t)-\bar{x}|
\leq\left(\frac{\sigma_{\rm max}(\Omega)}{\sigma_{\rm min}(\Omega)}
\sum_{j=1}^{p}|x_{j}(0)-\bar{x}|^{2}\right)^{1/2}
\end{eqnarray*}
where $\bar{x}:=(r^{T}\otimes I_{n})\ex(0)$.
\end{theorem} 

\subsection{Asymptotic synchronization}
We now begin looking for sufficient conditions on
$Q:\Real\to\overline{\Q}_{n}$ that guarantee that solutions
$x_{i}(\cdot)$ of array~\eqref{eqn:system} synchronize. The next fact
is to be used by the key theorem following it.

\begin{fact}\label{fact:int}
Let $f:[0,\,T]\to[0,\,1]$ be Riemann-integrable. Then 
\begin{eqnarray*}
3T^{2}\int_{0}^{T}f^{2}(t)dt\geq \left(\int_{0}^{T}f(t)dt\right)^{3}\,.
\end{eqnarray*}
\end{fact}

\begin{proof}
Result trivially follows for $T=0$. Suppose $T>0$.  Fix some
$\delta>0$ such that $T/\delta=:N$ is an integer and let
$I_{f}:=\int_{[0,\,T]}f$. Then there exists
$k_{1}\in\{1,\,2,\,\ldots,\,N\}$ and
$t_{1}\in[(k_{1}-1)\delta,\,k_{1}\delta]$ such that $f(t_{1})\geq
I_{f}/T$.  Since $f(t)\leq 1$ for all $t$, we can also claim that there
exists $k_{2}\in\{1,\,2,\,\ldots,\,N\}\setminus\{k_{1}\}$ and
$t_{2}\in[(k_{2}-1)\delta,\,k_{2}\delta]$ such that $f(t_{2})\geq
(I_{f}-\delta)/T$. Following the pattern, we can generate a sequence
$(k_{i})_{i=1}^{\lfloor I_{f}/\delta\rfloor}$ of distinct elements
from the set $\{1,\,2,\,\ldots,\,N\}$ and an associated sequence
$(t_{i})_{i=1}^{\lfloor I_{f}/\delta\rfloor}$ satisfying
$t_{i}\in[(k_{i}-1)\delta,\,k_{i}\delta]$ and $f(t_{i})\geq
(I_{f}-(i-1)\delta)/T$. Therefore we can write
\begin{eqnarray*}
\sum_{k=1}^{\lfloor T/\delta\rfloor}\sup_{t\in[(k-1)\delta,\,k\delta]}f^{2}(t)\delta
&\geq& \sum_{i=1}^{\lfloor I_{f}/\delta\rfloor}f^{2}(t_{i})\delta\\
&\geq& \delta
\sum_{i=1}^{\lfloor I_{f}/\delta\rfloor}\left(\frac{I_{f}-(i-1)\delta}{T}\right)^{2}\,.
\end{eqnarray*} 
By definition of integral, we can therefore write
\begin{eqnarray*}
\int_{0}^{T}f^{2}(t)dt
&=&\lim_{\delta\to 0^{+}}
\sum_{k=1}^{\lfloor T/\delta\rfloor}\sup_{t\in[(k-1)\delta,\,k\delta]}f^{2}(t)\delta\\
&\geq&\lim_{\delta\to 0^{+}} \delta
\sum_{i=1}^{\lfloor I_{f}/\delta\rfloor}\left(\frac{I_{f}-(i-1)\delta}{T}\right)^{2}\\
&=& \lim_{\delta\to 0^{+}}\frac{\delta}{T^{2}}
\sum_{i=1}^{\lfloor I_{f}/\delta\rfloor}(I_{f}-i\delta)^{2}\\
&=& \lim_{\delta\to 0^{+}}\frac{\delta}{T^{2}}\left(
\sum_{i=1}^{\lfloor I_{f}/\delta\rfloor}(I_{f}^{2}-2I_{f}i\delta)
+\sum_{i=1}^{\lfloor I_{f}/\delta\rfloor}i^{2}\delta^{2}
\right)\\
&=& \lim_{\delta\to 0^{+}}\frac{\delta^{3}}{T^{2}}
\sum_{i=1}^{\lfloor I_{f}/\delta\rfloor}i^{2}\\
&=& \frac{I_{f}^{3}}{3T^{2}}\,.
\end{eqnarray*}
Hence the result.
\end{proof}

\begin{theorem}\label{thm:main}
Given a pair of positive real numbers $(\varepsilon,\,T)$, define
\begin{eqnarray}\label{eqn:delta}
\delta(\varepsilon,\,T):=
\min\left\{\frac{\varepsilon}{2},\,\frac{\varepsilon^{3}}{240T^{5}}
\right\}\,.
\end{eqnarray}
Given connected interconnection $\Gamma\in\Real^{p\times p}$, let
$r\in\Real^{p}$ and symmetric positive definite matrix
$\Omega\in\Real^{p\times p}$ satisfy \eqref{eqn:r} and
\eqref{eqn:lyap}, respectively.  Define
\begin{eqnarray}\label{eqn:rho}
\rho(\Gamma):=\sigma_{\rm max}(\Omega)\max\{1,|\Gamma|^{3}\}\,.
\end{eqnarray}
Let $V(\ex):=\ex^{T}(\Omega\otimes I_{n})\ex$ for $\ex\in\Real^{np}$.
Then, for all $Q:\Real_{\geq 0}\to\overline{\Q}_{n}$, the below
inequality
\begin{eqnarray}\label{eqn:Q}
\sigma_{\rm min}\left(\int_{0}^{T}Q_{t}dt\right)\geq\varepsilon
\end{eqnarray} implies that  
solution $\ex(\cdot)$ of system~\eqref{eqn:systemkron}
satisfies
\begin{eqnarray*}
V(\ex(T)-\bar{\ex})
\leq \left(1-\frac{\delta(\varepsilon,\,T)}{\rho(\Gamma)}\right)V(\ex(0)-\bar{\ex})
\end{eqnarray*}
where $\bar{\ex}:=(\one{r^{T}}\otimes I_{n})\ex(0)$.
\end{theorem}

\begin{proof}
Given pair $(\varepsilon,\,T)$ let $\omega:=\varepsilon/(4T)$. 
Consider system~\eqref{eqn:systemkron}. Let us introduce
\begin{eqnarray*}
\xi(t):=\ex(t)-\bar{\ex}\,.
\end{eqnarray*}
By Lemma~\ref{lem:Vdot}, we have
\begin{eqnarray}\label{eqn:Vxi}
\dot{V}(\xi(t))=-\xi^{T}(t)(I_{p}\otimes Q_{t})\xi(t)\,.
\end{eqnarray}
Also, $\xi(\cdot)$ can be shown to satisfy
\begin{eqnarray}\label{eqn:xidot}
\dot{\xi}=(\Gamma\otimes Q_{t})\xi\,.
\end{eqnarray} 
Let $Q:\Real_{\geq 0}\to\overline{Q}_{n}$ satisfy \eqref{eqn:Q}. Then,
regarding the evolution of $\xi(\cdot)$, one of the two following cases
must be.

{\em Case~1:} $|\xi(t)-\xi(0)|\leq \omega|\xi(0)|$ for all
$t\in[0,\,T]$. Let $b(t):=\xi(t)-\xi(0)$ and recall that $|Q_{t}|\leq
1$. For reasons of economy, let us adopt the notation ${\bf
Q}_{t}:=(I_{p}\otimes Q_{t})$. Note that then we have $|{\bf
Q}_{t}|\leq 1$ as well as
\begin{eqnarray*}
\sigma_{\rm min}\left(\int_{0}^{T}{\bf Q}_{t}dt\right)\geq\varepsilon\,.
\end{eqnarray*}
(Observe that $T\geq \varepsilon$.) From \eqref{eqn:Vxi} we can write 
\begin{eqnarray}\label{eqn:case1}
V(\xi(T))
&=&V(\xi(0))-\int_{0}^{T}\xi^{T}(t){\bf Q}_{t}\xi(t)dt\nonumber\\
&=&V(\xi(0))-\int_{0}^{T}(\xi(0)+b(t))^{T}{\bf Q}_{t}(\xi(0)+b(t))dt\nonumber\\
&\leq&V(\xi(0))-\xi^{T}(0)\left[\int_{0}^{T}{\bf Q}_{t}dt\right]\xi(0)
-2\int_{0}^{T}b^{T}(t){\bf Q}_{t}\xi(0)dt-\int_{0}^{T}b^{T}(t){\bf Q}_{t}b(t)dt\nonumber\\
&\leq& V(\xi(0))-\varepsilon|\xi(0)|^2+2\omega T|\xi(0)|^{2}\nonumber\\
&=&V(\xi(0))-\frac{\varepsilon}{2}|\xi(0)|^{2}\nonumber\\
&\leq&\left(1-\frac{\varepsilon}{2\sigma_{\rm max}(\Omega)}\right)V(\xi(0))\,.
\end{eqnarray}

{\em Case~2:} $|\xi(\bar{t})-\xi(0)|=\omega|\xi(0)|$ for some
$\bar{t}\in(0,\,T]$. Without loss of generality, assume
$|\xi(t)-\xi(0)|<\omega|\xi(0)|$ for $t\in[0,\,\bar{t})$. We can by
\eqref{eqn:xidot} write
\begin{eqnarray}\label{eqn:1}
\int_{0}^{\bar{t}}|{\bf Q}_{t}\xi(t)|dt
&=&|\Gamma|^{-1}
\int_{0}^{\bar{t}}|\Gamma\otimes I_{n}||{\bf Q}_{t}\xi(t)|dt\nonumber\\
&\geq&|\Gamma|^{-1}
\int_{0}^{\bar{t}}|(\Gamma\otimes Q_{t})\xi(t)|dt\nonumber\\
&\geq&|\Gamma|^{-1}
\left|\int_{0}^{\bar{t}}(\Gamma\otimes Q_{t})\xi(t)dt\right|\nonumber\\
&=&|\Gamma|^{-1}\omega|\xi(0)|\,.
\end{eqnarray}
Since $|{\bf Q}_{t}\xi(t)|\leq (1+\omega)|\xi(0)|$ for
$t\in[0,\,\bar{t}]$, we can invoke Fact~\ref{fact:int} on
\eqref{eqn:1} and obtain
\begin{eqnarray}\label{eqn:lbound}
\int_{0}^{\bar{t}}|{\bf Q}_{t}\xi(t)|^{2}dt
&=&(1+\omega)^{2}|\xi(0)|^{2}\int_{0}^{\bar{t}}
\left(\frac{|{\bf Q}_{t}\xi(t)|}{(1+\omega)|\xi(0)|}\right)^{2}dt\nonumber\\
&\geq&(1+\omega)^{2}|\xi(0)|^{2}\left(\frac{\omega}{(1+\omega)|\Gamma|}\right)^{3}
\frac{1}{3\bar{t}^{2}}\nonumber\\
&\geq&(1+\omega)^{2}|\xi(0)|^{2}\left(\frac{\omega}{(1+\omega)|\Gamma|}\right)^{3}
\frac{1}{3T^{2}}\nonumber\\
&=& \frac{\omega^{3}}{3(1+\omega)T^{2}|\Gamma|^{3}}|\xi(0)|^{2}\,.
\end{eqnarray}
Since $|{\bf Q}_{t}|\leq 1$, we have ${\bf Q}^{1/2}_{t}\geq{\bf
Q}_{t}$. Also, note that $V(\xi(t))$ is nonincreasing thanks to
\eqref{eqn:Vxi}. Now, by \eqref{eqn:lbound} we can write
\begin{eqnarray}\label{eqn:case2}
V(\xi(T))
&=&V(\xi(0))-\int_{0}^{T}\xi^{T}(t){\bf Q}_{t}\xi(t)dt\nonumber\\
&\leq&V(\xi(0))-\int_{0}^{\bar{t}}\xi^{T}(t){\bf Q}_{t}\xi(t)dt\nonumber\\
&=&V(\xi(0))-\int_{0}^{\bar{t}}|{\bf Q}^{1/2}_{t}\xi(t)|^{2}dt\nonumber\\
&\leq&V(\xi(0))-\int_{0}^{\bar{t}}|{\bf Q}_{t}\xi(t)|^{2}dt\nonumber\\
&\leq&V(\xi(0))-\frac{\omega^{3}}{3(1+\omega)T^{2}|\Gamma|^{3}}|\xi(0)|^{2}
\nonumber\\
&\leq&\left(1-
\frac{\varepsilon^{3}}{48T^{4}(4T+\varepsilon)|\Gamma|^{3}\sigma_{\rm max}(\Omega)}
\right)V(\xi(0))\nonumber\\
&\leq&\left(1-
\frac{\varepsilon^{3}}{240T^{5}|\Gamma|^{3}\sigma_{\rm max}(\Omega)}
\right)V(\xi(0))\,.
\end{eqnarray}
The result follows by \eqref{eqn:delta}, \eqref{eqn:rho},
\eqref{eqn:case1}, and \eqref{eqn:case2}.
\end{proof}

\begin{fact}\label{fact:sum}
Let $(a_{i})_{i=1}^{\infty}$ be a sequence with $0\leq a_{i}<1$ for
all $i$. Then, product $\prod_{i=1}^{\infty}(1-a_{i})$ converges to
zero if and only if sum $\sum_{i=1}^{\infty}a_{i}$ diverges.
\end{fact}

\begin{proof}
See e.g. \cite[Thm. 1.17 of Ch. VII]{SakZy-65}.
\end{proof}

In the light of Theorem~\ref{thm:main} and Fact~\ref{fact:sum}, we now
state our most general condition on
$Q:\Real\to\overline{\Q}_{n}$ for synchronization of solutions
$x_{i}(\cdot)$ of array~\eqref{eqn:system}.

\begin{definition}\label{def:suff}
Function $Q:\Real_{\geq 0}\to\Q_{n}$ is said to be 
{\em sufficiently exciting} if there exists a
sequence of pairs of positive real numbers
$(\varepsilon_{i},\,T_{i})_{i=1}^{\infty}$ satisfying
\begin{eqnarray}\label{eqn:ei}
\sigma_{\rm min}\left(\int_{t_{i}}^{t_{i}+T_{i}}
Q_{t}dt\right)\geq\varepsilon_{i}
\end{eqnarray}
for $t_{i}=\sum_{j=1}^{i-1}T_{j}$ with $t_{1}=0$, and
\begin{eqnarray}\label{eqn:di}
\lim_{N\to\infty}\sum_{i=1}^{N}\delta(\varepsilon_{i},\,T_{i})=\infty
\end{eqnarray}
where $\delta(\cdot,\,\cdot)$ is as defined in \eqref{eqn:delta}.
\end{definition}

\begin{theorem}\label{thm:suff}
Let $Q:\Real_{\geq 0}\to\overline{\Q}_{n}$ be sufficiently exciting.
Then, for all connected interconnection $\Gamma\in\Real^{p\times p}$,
solutions $x_{i}(\cdot)$ of array~\eqref{eqn:system} synchronize to
\begin{eqnarray*}
\bar{x}(t)\equiv(r^{T}\otimes I_{n})\ex(0)
\end{eqnarray*}
where $r\in\Real^{p}$ satisfies \eqref{eqn:r}.
\end{theorem}

\begin{proof}
Let us be given function $Q:\Real_{\geq 0}\to\overline{Q}_{n}$ and
sequence $(\varepsilon_{i},\,T_{i})$ satisfying \eqref{eqn:ei} and
\eqref{eqn:di}. Let us let $\delta_{i}:=\delta(\varepsilon_{i},\,T_{i})$. 
Let interconnection $\Gamma\in\Real^{p\times p}$ be connected,
$r\in\Real^{p}$ satisfy \eqref{eqn:r}, and symmetric positive
definite matrix $\Omega\in\Real^{p\times p}$ satisfy
\eqref{eqn:lyap}. Consider system~\eqref{eqn:systemkron} and let
$\bar{\ex}=(\one r^{T}\otimes I_{n})\ex(0)$. Now, letting
$V(\ex)=\ex^{T}(\Omega\otimes I_{n})\ex$, by Theorem~\ref{thm:main} we
can write
\begin{eqnarray*}
V(\ex(t_{i+1})-\bar{\ex})\leq\left(1-\frac{\delta_{i}}{\rho(\Gamma)}\right)V(\ex(t_{i})-\bar{\ex})
\end{eqnarray*} 
for $i=1,\,2,\,\ldots$ where $\rho(\cdot)$ is as defined in
\eqref{eqn:rho}. Whence
\begin{eqnarray}\label{eqn:Vti}
V(\ex(t_{i})-\bar{\ex})\leq V(\ex(0)-\bar{\ex})
\prod_{j=1}^{i-1}\left(1-\frac{\delta_{j}}{\rho(\Gamma)}\right)\,.
\end{eqnarray}
Let now $a_{i}:=\delta_{i}/\rho(\Gamma)$. We can write by \eqref{eqn:di}
\begin{eqnarray*}
\lim_{N\to\infty}\sum_{i=1}^{N}a_{i}
&=&\rho(\Gamma)^{-1}\lim_{N\to\infty}\sum_{i=1}^{N}\delta_{i}\\
&=&\infty\,.
\end{eqnarray*}
Now, we can invoke Fact~\ref{fact:sum} and claim that
$\lim_{N\to\infty}\prod_{i=1}^{N}(1-a_{i})=0$, which yields by
\eqref{eqn:Vti}
\begin{eqnarray*}
\lim_{i\to\infty}V(\ex(t_{i})-\bar{\ex})=0\,.
\end{eqnarray*}
Hence the result.
\end{proof}

\begin{corollary}\label{cor:general}
Let $Q:\Real_{\geq 0}\to\overline{\Q}_{n}$ be sufficiently exciting. 
Then solution to linear system $\dot{x}=-Q_{t}x$ satisfies 
$\lim_{t\to\infty}x(t)=0$. 
\end{corollary}

The following definition is quite standard; see, for
instance, \cite{loria05}.  Note that the condition it depicts is less
general than that of Definition~\ref{def:suff}, yet it guarantees
exponential synchronization.

\begin{definition}
Map $Q:\Real_{\geq 0}\to\Q_{n}$ is said to be {\em persistently exciting}
if there exists a pair of positive real
numbers $(\varepsilon,\,T)$ such that
\begin{eqnarray}\label{eqn:uniform}
\sigma_{\rm min}\left(\int_{t}^{t+T}Q_{\tau}d\tau\right)\geq\varepsilon
\end{eqnarray}
for all $t\geq 0$.
\end{definition}

\begin{remark}
The following theorem can be viewed as a generalization of a classic
result in adaptive control theory \cite[Thm.~2.5.1]{sastry89}.  Note that
Corollary~\ref{cor:general} makes another generalization to this result
since ``sufficiently exciting'' is weaker than ``persistently exciting''.
\end{remark}

\begin{theorem}\label{thm:suffexp}
Let interconnection $\Gamma\in\Real^{p\times p}$ be connected and 
function $Q:\Real_{\geq 0}\to \overline{\Q}_{n}$ be persistently exciting. Then,
solutions $x_{i}(\cdot)$ of array~\eqref{eqn:system} exponentially
synchronize to
\begin{eqnarray*}
\bar{x}(t)\equiv(r^{T}\otimes I_{n})\ex(0)
\end{eqnarray*}
where $r\in\Real^{p}$ satisfies \eqref{eqn:r}.
\end{theorem}

\begin{proof}
Since $Q:\Real_{\geq 0}\to \overline{\Q}_{n}$ is persistently
exciting, by definition, there exists a pair of positive real numbers
$(\varepsilon,\,T)$ satisfying \eqref{eqn:uniform} for all $t\geq
0$. Let $V(\ex)=\ex^{T}(\Omega\otimes I_{n})\ex$ for
$\ex\in\Real^{np}$, where symmetric positive definite matrix
$\Omega\in\Real^{p\times p}$ satisfy \eqref{eqn:lyap}. Then, by
Theorem~\ref{thm:main} we can write
\begin{eqnarray*}
V(\ex(kT)-\bar{\ex})
\leq\left(1-\frac{\delta(\varepsilon,\,T)}{\rho(\Gamma)}\right)^{k}V(\ex(0)-\bar{\ex})
\end{eqnarray*}
for all $k\in\Natural$, where $\bar{\ex}:=(\one r^{T}\otimes
I_{n})\ex(0)$. The result then follows.
\end{proof}

We now present an interesting application of Theorem~\ref{thm:suffexp}
on coupled harmonic oscillators (in $\Real^{2}$) described by
\begin{eqnarray*}
\dot{x}_{i1}&=&x_{i2}\\
\dot{x}_{i2}&=& -x_{i1}+\sum_{j\neq i}\gamma_{ij}(x_{j2}-x_{i2})
\end{eqnarray*}
for $i=1,\,2,\,\ldots,\,p$. Let
\begin{eqnarray*}
x_{i}=\left[\!\!\begin{array}{r}x_{i1}\\x_{i2}\end{array}\!\!\right]\ ,\quad
A = \left[\!\!\begin{array}{rr}0 & 1\\-1&0\end{array}\!\!\right]\ ,\quad
C = [0\ 1]\,.
\end{eqnarray*}
Then we can write
\begin{eqnarray*}
\dot{x}_{i}=Ax_{i}+C^{T}C\sum_{j\neq i}\gamma_{ij}(x_{j}-x_{i})\,.
\end{eqnarray*} 
Define $\xi_{i}(t):=e^{-At}x_{i}(t)$ and $\xi=[\xi_{1}^{T}\ \ldots\
\xi_{p}^{T}]^{T}$. Then
\begin{eqnarray*}
\dot{\xi}=(\Gamma\otimes e^{A^{T}t}C^{T}Ce^{At})\xi
\end{eqnarray*}
for $A$ is skew-symmetric. A trivial computation shows
\begin{eqnarray*}
e^{A^{T}t}C^{T}Ce^{At}=
\left[\!\!\begin{array}{rr}\cos^{2}t & -\sin t\cos t\\
-\sin t\cos t&\sin^{2}t\end{array}\!\!\right]\in{\overline\Q}_{2}
\end{eqnarray*}
whence
\begin{eqnarray*}
\int_{t}^{t+2\pi}e^{A^{T}\tau}C^{T}Ce^{A\tau}d\tau=\pi I_{2}
\end{eqnarray*}
for all $t$. Therefore $t\mapsto e^{A^{T}t}C^{T}Ce^{At}$ is
persistently exciting.  Now, suppose that $\Gamma$ is connected and
$r\in\Real^{p}$ satisfies \eqref{eqn:r}. Then, by
Theorem~\ref{thm:suffexp} solutions $\xi_{i}(\cdot)$ exponentially
synchronize to $\bar\xi(t)\equiv (r^{T}\otimes I_{2})\xi(0)$. Since
$x_{i}(t)=e^{At}\xi_{i}(t)$ and $e^{At}$ is an orthogonal (hence
norm-preserving) matrix, it follows that solutions $x_{i}(\cdot)$ of
the coupled harmonic oscillators exponentially synchronize to
\begin{eqnarray*}
\bar{x}(t)=(r^{T}\otimes e^{At})\left[\!\! \begin{array}{c}x_{1}(0)\\
\vdots\\x_{p}(0)\end{array}\!\!\right]
\end{eqnarray*}

\subsection{Negative results}\label{subsec:negative}

Before we end this section, we present two negative results, which we
believe constitute answers to questions that arise naturally. The
first of those questions emerges as follows.  In
Theorem~\ref{thm:stability} we have proven that system
$\dot{\ex}=(\Gamma\otimes Q_{t})\ex$, where $\Gamma$ is a {\em fixed}
interconnection and $Q$ is a time-varying SPSD matrix, has a bounded
solution for all initial conditions. It also trivially follows from
the results in, for instance,
\cite{moreau05,lin07} that solution of system
$\dot{\ex}=(\Gamma_{t}\otimes Q)\ex$, where this time interconnection
$\Gamma$ is time-varying and SPSD matrix $Q$ is fixed, is bounded. At
this point, it is tempting to ask the following question.
\begin{center}
{\em Is solution of system $\dot{\ex}=(\Gamma_{t}\otimes Q_{t})\ex$ bounded?}
\end{center}
The answer is {\em not always} and it is formalized in the below result.

\begin{theorem}\label{thm:neg1}
There exist maps $\Gamma:\Real\to\Real^{p\times p}$, where
$\Gamma_{t}$ is an interconnection for each $t$, and
$Q:\Real\to\overline{\Q}_{n}$ such that system
$\dot{\ex}=(\Gamma_{t}\otimes Q_{t})\ex$ has an {\em unbounded}
solution.
\end{theorem}

\begin{proof}
We construct $\Gamma$ and $Q$ as follows. Let
\begin{eqnarray*}
\Gamma_{\rm a}:=
\left[
\begin{array}{rrrr}
-1&1&0&0\\
0&0&0&0\\
0&0&0&0\\
0&0&1&-1
\end{array}
\right]\ ,\quad
\Gamma_{\rm b}:=
\left[
\begin{array}{rrrr}
0&0&0&0\\
0&-1&0&1\\
1&0&-1&0\\
0&0&0&0
\end{array}
\right]\\
\Gamma_{\rm c}:=
\left[
\begin{array}{rrrr}
-1&0&1&0\\
0&0&0&0\\
0&0&0&0\\
0&1&0&-1
\end{array}
\right]\ ,\quad
\Gamma_{\rm d}:=
\left[
\begin{array}{rrrr}
0&0&0&0\\
1&-1&0&0\\
0&0&-1&1\\
0&0&0&0
\end{array}
\right]
\end{eqnarray*}
and
\begin{eqnarray*}
Q_{\rm a}:=
\left[
\begin{array}{rr}
1&0\\0&0
\end{array}
\right]\ ,\quad
Q_{\rm b}:=
\left[
\begin{array}{rr}
0&0\\0&1
\end{array}
\right]\ ,\quad
Q_{\rm c}:=
\left[
\begin{array}{rr}
0.5&0.5\\0.5&0.5
\end{array}
\right]\ ,\quad
Q_{\rm d}:=
\left[
\begin{array}{rr}
0.5&-0.5\\-0.5&0.5
\end{array}
\right]
\end{eqnarray*}
Now let both $\Gamma$ and $Q$ be periodic with period $T=40$ with 
\begin{eqnarray*}
\Gamma_{t}:=\left\{
\begin{array}{lcr}
\Gamma_{\rm a}&\mbox{for}&0\leq t < 10\\
\Gamma_{\rm b}&\mbox{for}&10\leq t < 20\\
\Gamma_{\rm c}&\mbox{for}&20\leq t < 30\\
\Gamma_{\rm d}&\mbox{for}&30\leq t < 40
\end{array}
\right.\quad
\mbox{and}\quad
Q_{t}:=\left\{
\begin{array}{lcr}
Q_{\rm a}&\mbox{for}&0\leq t < 10\\
Q_{\rm b}&\mbox{for}&10\leq t < 20\\
Q_{\rm c}&\mbox{for}&20\leq t < 30\\
Q_{\rm d}&\mbox{for}&30\leq t < 40
\end{array}
\right.
\end{eqnarray*}
Whence we can write for $k=0,\,1,\,\ldots$
\begin{eqnarray*}
\ex(kT)={\mathbf A}^{k}
\ex(0)
\end{eqnarray*}
where ${\mathbf A}:=e^{(\Gamma_{\rm d}\otimes Q_{\rm
d})10}e^{(\Gamma_{\rm c}\otimes Q_{\rm c})10}e^{(\Gamma_{\rm b}\otimes
Q_{\rm b})10}e^{(\Gamma_{\rm a}\otimes Q_{\rm a})10}$. When the
eigenvalues of ${\mathbf A}$ are numerically checked, one finds that
there is an eigenvalue outside the unit circle ($|\lambda|\approx 2$),
which lets us deduce that the origin of system
$\dot{\ex}=(\Gamma_{t}\otimes Q_{t})\ex$ is unstable.
\end{proof}

Our first question was concerned with stability under time-varying
interconnection; and we have seen that solutions $x_{i}(\cdot)$ of
array~\eqref{eqn:system} need not stay bounded in such a case. The
second question is about synchronization. By Theorem~\ref{thm:suff} we
know that if map $Q:\Real_{\geq 0}\to
\overline{\Q}_{n}$ is sufficiently exciting, then for all connected
interconnection $\Gamma$ solutions of array~\eqref{eqn:system}
synchronize. Now, suppose that we are given some sufficiently exciting
$Q$ with an associated sequence
$(\varepsilon_{i},\,T_{i})_{i=1}^{\infty}$, see
Definition~\ref{def:suff}. Note that, due to \eqref{eqn:delta}, we have
$\sum_{i=1}^{\infty}\varepsilon_{i}=\infty$. In addition, since
$Q_{t}\in\overline{\Q}_{n}$ for all $t$, we have
$T_{i}\geq\varepsilon_{i}$, which yields
$\sum_{i=1}^{\infty}T_{i}=\infty$. Applying these observations on
\eqref{eqn:ei}, we obtain
\begin{eqnarray}\label{eqn:issuff}
\lim_{T\to\infty}\sigma_{\rm min}\left(\int_{0}^{T}Q_{t}dt\right)=\infty\,.
\end{eqnarray}
Above condition, depicted in \eqref{eqn:issuff}, can be shown to be necessary
for synchronization. Now, we ask the following question. 
\begin{center}
{\em Is condition~\eqref{eqn:issuff} sufficient for synchronization?}
\end{center}
The answer turns out to be negative.  In fact, even a
much stronger condition is not sufficient for synchronization
as the following result shows.

\begin{theorem}\label{thm:neg2}
There exist connected interconnection $\Gamma\in\Real^{p\times p}$ and
map $Q:\Real\to \overline{\Q}_{n}$ satisfying
\begin{eqnarray*}
\liminf_{T\to\infty}\frac{1}{T}\ \sigma_{\rm min}\left(\int_{0}^{T}Q_{t}dt\right)>0
\end{eqnarray*}
such that solutions $x_{i}(\cdot)$ of array~\eqref{eqn:system} do
{\em not} synchronize.
\end{theorem}

\begin{proof}
As in the proof of the previous result, we will once again make use of
projection matrices. Let $\varepsilon_{k}=2^{-k}$ for
$k=1,\,2,\,\ldots$ Then, let piecewise linear function
$\vartheta:\Real_{\geq 0}\to\Real$ be
\begin{eqnarray*}
\vartheta(t):=\vartheta(\tau_{k-1})+t\sin\varepsilon_{k}\cos\varepsilon_{k}\quad
\mbox{for}\quad t\in[\tau_{k-1},\,\tau_{k})
\end{eqnarray*}
where 
$\tau_{0}=0$, $\vartheta(0)=0$, and
\begin{eqnarray*}
\tau_{k}&=&\tau_{k-1}+\frac{2\pi}{\sin\varepsilon_{k}\cos\varepsilon_{k}}\\
\vartheta(\tau_{k})&=&\lim_{t\to\tau_{k}^{-}}\vartheta(t)-\varepsilon_{k+1}
\end{eqnarray*}
for $k=1,\,2,\,\ldots$
Now define $Q:\Real_{\geq 0}\to
\overline\Q_{2}$ as 
\begin{eqnarray*}
Q_{t}:=\left[
\begin{array}{rr}
\cos^2\vartheta(t)&\sin\vartheta(t)\cos\vartheta(t)\\
\sin\vartheta(t)\cos\vartheta(t)&\sin^{2}\vartheta(t)
\end{array}
\right]
\end{eqnarray*}
which is a projection matrix that projects onto the line spanned by
the vector $[\cos\vartheta(t)\ \sin\vartheta(t)]^{T}$. Calculations yield
\begin{eqnarray*}
\frac{1}{\tau_{k}-\tau_{k-1}}
\int_{\tau_{k-1}}^{\tau_{k}}Q_{t}dt=\frac{1}{2}I_{2}\,.
\end{eqnarray*}
Observe that 
\begin{eqnarray}\label{eqn:nokia}
\lim_{k\to\infty}\frac{\tau_{k-1}}{\tau_{k}}=\frac{1}{2}
\end{eqnarray}
and 
\begin{eqnarray*}
\sigma_{\rm min}\left(\alpha I_{2}+\int_{t_{1}}^{t_{2}}Q_{t}dt\right)=
\alpha+\sigma_{\rm min}\left(\int_{t_{1}}^{t_{2}}Q_{t}dt\right)
\end{eqnarray*}
for all $\alpha\geq 0$ and $t_{2}\geq t_{1}\geq 0$. Given
$T\in(\tau_{k-1},\,\tau_{k}]$ we can write 
\begin{eqnarray*}
\frac{1}{T}\ \sigma_{\rm min}\left(\int_{0}^{T}Q_{t}dt\right)
&=& \frac{1}{T}\ \sigma_{\rm
min}\left(\int_{0}^{\tau_{k-1}}Q_{t}dt+\int_{\tau_{k-1}}^{T}Q_{t}dt\right)\\
&=& \frac{1}{T}\ \sigma_{\rm
min}\left(\frac{\tau_{k-1}}{2}I_{2}+\int_{\tau_{k-1}}^{T}Q_{t}dt\right)\\
&\geq& \frac{\tau_{k-1}}{2T}\\
&\geq& \frac{\tau_{k-1}}{2\tau_{k}}
\end{eqnarray*}
which yields by \eqref{eqn:nokia} that
\begin{eqnarray*}
\liminf_{T\to\infty}\frac{1}{T}\ \sigma_{\rm min}\left(\int_{0}^{T}Q_{t}dt\right)
\geq\frac{1}{4}\,.
\end{eqnarray*}

Let us now consider \eqref{eqn:system} under the following connected
interconnection
\begin{eqnarray*}
\Gamma:=\left[
\begin{array}{rr}
-1 & 1 \\
 0 & 0
\end{array}
\right]
\end{eqnarray*}
Setting $x_{2}(0)=0$ we
can write
\begin{eqnarray}\label{eqn:euclidean} 
\dot{x}_{1}=-Q_{t}x_{1}\,.
\end{eqnarray}
Note that we need $\lim_{t\to\infty}x_{1}(t)=0$ for synchronization
since $x_{2}(\cdot)\equiv 0$.  In terms of polar coordinates,
i.e. $x_{1}=[r\cos\theta\ r\sin\theta]^{T}$, we can express \eqref{eqn:euclidean}
as
\begin{subeqnarray}\label{eqn:rt}
\dot{r}&=&-r\sin^{2}(\vartheta(t)-\theta-\pi/2)\\
\dot{\theta}&=&\sin(\vartheta(t)-\theta-\pi/2)\cos(\vartheta(t)-\theta-\pi/2)\,.
\end{subeqnarray}
Let us initialize $x_{1}$ such that $r>0$ and
$\theta(0)=-\pi/2-\varepsilon_{1}$. We then observe that
$\dot{\theta}(t)=\dot{\vartheta}(t)$ for all $t\geq 0$. Eq.~\eqref{eqn:rt}
simplifies to 
\begin{eqnarray*}
\dot{r}&=&-r\sin^{2}\varepsilon_{k}\\
\dot{\theta}&=&\sin\varepsilon_{k}\cos\varepsilon_{k}
\end{eqnarray*}
for $t\in[\tau_{k-1},\,\tau_{k})$ and $k=1,\,2,\,\ldots$ Thence 
\begin{eqnarray*}
r(\tau_{k})=r(\tau_{k-1})e^{-(\tau_{k}-\tau_{k-1})\sin^{2}\varepsilon_{k}}
\end{eqnarray*}
which yields
\begin{eqnarray*}
r(\tau_{k})
&=&r(0)\exp\left(-2\pi\sum_{i=1}^{k}\tan\varepsilon_{i}\right)\\
&=&r(0)\exp\left(-2\pi\sum_{i=1}^{k}\tan2^{-i}\right)\\
&\geq&r(0)\exp\left(-4\pi\sum_{i=1}^{k}2^{-i}\right)\\
&\geq&r(0)e^{-4\pi}\
\end{eqnarray*}
for all $k$. Therefore $\lim_{t\to\infty}x_{1}(t)\neq 0$. 
\end{proof}

\section{Observability grammian and synchronizability}\label{sec:main}

Based on the results of the previous section, we are now ready to
establish our theorems that are aimed to reveal the correlation
between synchronizability and observability grammian. We begin with
two definitions.

\begin{definition}\label{def:asyobs}
For $A:\Real\to\Real^{n\times n}$ and $C:\Real\to\Real^{m\times n}$,
pair $(C,\,A)$ is said to be {\em asymptotically observable} if 
the integrand of 
the observability grammian, $t\mapsto \Phi_{A}^{T}(t,\,0)C^{T}(t)C(t)
\Phi_{A}(t,\,0)$, is sufficiently exciting.
\end{definition}

The following definition is borrowed (with slight modification) from
\cite{chen99}.

\begin{definition}\label{def:uniobs}
For $A:\Real\to\Real^{n\times n}$ and $C:\Real\to\Real^{m\times n}$,
pair $(C,\,A)$ is said to be {\em uniformly observable} if there
exists a pair of positive real numbers $(\varepsilon,\,T)$ such that 
\begin{eqnarray*}
\sigma_{\rm min}(W_{\rm o}(t,\,t+T))\geq\varepsilon
\end{eqnarray*}
for all $t\geq 0$.
\end{definition}

\begin{remark}\label{rem:obsdef}
For $A:\Real\to\Real^{n\times n}$ and $C:\Real\to\Real^{m\times n}$
satisfying Assumption~\ref{assume:AC}, asymptotic observability of
pair $(C,\,A)$ implies uniform observability of $(C,\,A)$.  For a
time-invariant pair, which need {\em not} satisfy
Assumption~\ref{assume:AC}; asymptotic observability, uniform
observability, and the standard definition of observability (for
time-invariant linear systems) are all equivalent.
\end{remark}

The next result is our main theorem. It states that a time-varying
pair $(C,\,A)$ is synchronizable if it is asymptotically observable.

\begin{theorem}\label{thm:ACmain}
Let $A:\Real\to\Real^{n\times n}$ and $C:\Real\to\Real^{m\times n}$
satisfy Assumption~\ref{assume:AC}. If pair $(C,\,A)$ is asymptotically
observable, then it is synchronizable. In particular, if we choose
$L:\Real\to\Real^{n\times n}$ as in \eqref{eqn:L}, then for each
$\Gamma\in\G_{>0}$, solutions $x_{i}(\cdot)$ of array~\eqref{eqn:array}
with $u_{i}=L(t)z_{i}$ synchronize to
\begin{eqnarray}\label{eqn:arnold}
\bar{x}(t):=(r^{T}\otimes\Phi_{A}(t,\,0))
\left[\!\!
\begin{array}{c}
x_{1}(0)\\
\vdots\\
x_{p}(0)
\end{array}\!\!
\right]
\end{eqnarray}
where $r\in\Real^{p}$ satisfies \eqref{eqn:r}.
\end{theorem}

\begin{proof}
Let $L$ be as in \eqref{eqn:L}; then it is bounded by
Assumption~\ref{assume:AC}. Let $\Gamma\in\Real^{p\times p}$ be a
connected interconnection and $r\in\Real^{p}$ satisfy
\eqref{eqn:r}. Consider array~\eqref{eqn:array} with
$u_{i}=L(t)z_{i}$. Define auxiliary variables $\xi_{i}$ as
\begin{eqnarray}\label{eqn:grothendieck}
\xi_{i}(t)=\Phi_{A}(0,\,t)x_{i}(t)\,.
\end{eqnarray}
Then, we can write
\begin{eqnarray}\label{eqn:xibarabarc}
\dot{\xi}_{i}=Q_{t}\sum_{j\neq i}\hat\gamma_{ij}(\xi_{j}-\xi_{i})
\end{eqnarray}
where
$Q_{t}:=(\bar{a}\bar{c})^{-1}\Phi_{A}^{T}(t,\,0)C^{T}(t)C(t)\Phi_{A}(t,\,0)$,
$\hat\gamma_{ij}:=\bar{a}\bar{c}\gamma_{ij}$, and
$\bar{a},\,\bar{c}\geq 1$ come from Assumption~\ref{assume:AC}. Note
that $Q_{t}\in\overline{\Q}_{n}$ for all $t\geq 0$, $t\mapsto Q_{t}$
is sufficiently exciting, and $\widehat\Gamma:=[\hat\gamma_{ij}]$
is a connected interconnection satisfying
$r^{T}\widehat\Gamma=0$. We now invoke Theorem~\ref{thm:suff} on
\eqref{eqn:xibarabarc} to deduce that 
solutions $\xi_{i}(\cdot)$ synchronize to 
\begin{eqnarray}\label{eqn:rene}
\bar{\xi}(t)\equiv(r^{T}\otimes I_{n})
\left[\!\!
\begin{array}{c}
x_{1}(0)\\
\vdots\\
x_{p}(0)
\end{array}\!\!
\right]
\end{eqnarray}
Recall that $\Phi_{A}$ is bounded. Hence, 
combining \eqref{eqn:rene} and \eqref{eqn:grothendieck} yields
\eqref{eqn:arnold}. 
\end{proof}

As noted in Remark~\ref{rem:obsdef}, uniform observability is more
restrictive a condition than asymptotic observability. However, it has
a stronger outcome as stated by the following theorem.

\begin{theorem}\label{thm:ACexp}
Let $A:\Real\to\Real^{n\times n}$ and $C:\Real\to\Real^{m\times n}$
satisfy Assumption~\ref{assume:AC}. Then pair $(C,\,A)$ is
synchronizable if it is uniformly observable. In particular, if we
choose $L:\Real\to\Real^{n\times n}$ as in \eqref{eqn:L}, then for each
$\Gamma\in\G_{>0}$, solutions $x_{i}(\cdot)$ of array~\eqref{eqn:array}
with $u_{i}=L(t)z_{i}$ exponentially synchronize to
\begin{eqnarray*}
\bar{x}(t):=(r^{T}\otimes\Phi_{A}(t,\,0))
\left[\!\!
\begin{array}{c}
x_{1}(0)\\
\vdots\\
x_{p}(0)
\end{array}\!\!
\right]
\end{eqnarray*}
where $r\in\Real^{p}$ satisfies \eqref{eqn:r}.
\end{theorem}

\begin{proof}
The demonstration flows in a way that is analogous to that of
Theorem~\ref{thm:ACmain}. This time, however, the result follows from
Theorem~\ref{thm:suffexp}.
\end{proof}

\section{Discrete-time results}\label{sec:dt}
Our study on synchronization has hitherto been solely in continuous
time. However, it is possible to extend the analysis to systems in
discrete time without much difficulty. In fact, most of the
continuous-time results have discrete-time counterparts under
assumptions and definitions that are analogous to the ones we used for
continuous-time systems. In this section, therefore, we focus on
discrete-time time-varying linear systems and investigate the
correlation between synchronizability and observability grammian in
discrete time.

For a given interconnection $\Lambda=[\lambda_{ij}]\in\Real^{p\times p}$, consider
the array of $p$ discrete-time linear systems, for $k\in\Natural$,
\begin{subeqnarray}\label{eqn:arraydt}
x_{i}^{+}&=&A(k)x_{i}+u_{i}\\
y_{i}&=&C(k)x_{i}\\
z_{i}&=&\sum_{j\neq i}\lambda_{ij}(y_{j}-y_{i})
\end{subeqnarray}  
where $x_{i}\in\Real^{n}$ is the state, $x_{i}^{+}$ is the state at
the next time instant, $u_{i}\in\Real^{n}$ is the input,
$y_{i}\in\Real^{m}$ is the output, and $z_{i}\in\Real^{m}$ is the
coupling of the $i$th system for $i=1,\,2,\,\ldots,\,p$. For each
$k\in\Natural$, we have $A(k)\in\Real^{n\times n}$ and
$C(k)\in\Real^{m\times n}$. The solution of $i$th system at time
$k\in\Natural$ is denoted by $x_{i}(k)$. We denote by
$\Phi_{A}(\cdot,\,\cdot)$ the state transition matrix for $A$,
i.e. for $k>k_{0}$
\begin{eqnarray*}
\Phi_{A}(k,\,k_{0})=A(k-1)A(k-2)\cdots A(k_{0})
\end{eqnarray*} 
with $\Phi_{A}(k_{0},\,k_{0})=I_{n}$. We will let
$\Phi_{A}(k_{0},\,k)=\Phi^{-1}_{A}(k,\,k_{0})$ whenever the inverse
exists.  Observability grammian for pair $(C,\,A)$ is given by
\begin{eqnarray*}
W_{\rm o}(k_{0},\,k):=\sum_{\ell=k_{0}}^{k-1}\Phi_{A}^{T}(\ell,\,k_{0})C^{T}(\ell)C(\ell)
\Phi_{A}(\ell,\,k_{0})
\end{eqnarray*}
for $k,\,k_{0}\in\Natural$. Below we provide the discrete-time
versions of Definition~\ref{def:sync} and Assumption~\ref{assume:AC}.

\begin{definition}[Synchronizability]
Given functions $A:\Natural\to\Real^{n\times n}$ and
$C:\Natural\to\Real^{m\times n}$; pair $(C,\,A)$ is said to be {\em
synchronizable (with respect to $\G_{>0}$)} if there exists a bounded,
time-varying linear feedback law $L:\Natural\to\Real^{n\times m}$ such
that for each $\Lambda\in\G_{>0}$, solutions $x_{i}(\cdot)$ of
array~\eqref{eqn:arraydt} with $u_{i}=L(k)z_{i}$ synchronize for all
initial conditions.
\end{definition}

\begin{assumption}[Boundedness]\label{assume:ACdt}
For $A:\Natural\to\Real^{n\times n}$ and $C:\Natural\to\Real^{m\times n}$
following hold.
\begin{itemize}
\item[{(a)}] For each $k\in\Natural$, $A^{-1}(k)$ exists. 
There exists $\bar{a}\geq 1$ such that 
$|\Phi_{A}(k_{1},\,k_{2})|\leq \bar{a}$ for all $k_{1},\,k_{2}\in\Natural$.
\item[{(b)}] There exists $\bar{c}\geq 1$ such that $|C(k)|\leq \bar{c}$ 
for all $k\in\Natural$.
\end{itemize}
\end{assumption}

\begin{remark}
When $A$ and $C$ are constant matrices,
Assumption~\ref{assume:ACdt}(b) comes for free; and
Assumption~\ref{assume:ACdt}(a) becomes equivalent to that all
eigenvalues of matrix $A$ are with unity magnitude and none of them
belongs to a Jordan block with size two or greater.
\end{remark}

\subsection{Synchronization under bounded SPSD matrix} 

This subsection will emulate Section~\ref{sec:spsd}, where we studied
the stability and synchronization properties of
array~\eqref{eqn:system}. For a given interconnection
$\Lambda=[\lambda_{ij}]\in\Real^{p\times p}$, let an array of $p$
systems be
\begin{eqnarray}\label{eqn:systemdt}
x_{i}^{+}=x_{i}+Q_{k}\sum_{j\neq i}\lambda_{ij}(x_{j}-x_{i})
\end{eqnarray}
where $x_{i}\in\Real^{n}$ and $Q_{k}\in\overline{\Q}_{n}$ for all
$k\in\Natural$. We consider \eqref{eqn:systemdt} as the discrete-time
analogue of \eqref{eqn:system}. Let us stack individual vectors
$x_{i}$ into $\ex=[x_{1}^{T}\ x_{2}^{T}\ \ldots\ x_{p}^{T}]^{T}$. Then
we obtain from \eqref{eqn:systemdt} 
\begin{eqnarray}\label{eqn:systemkrondt}
\ex^{+}=(I_{np}+(\Lambda-I_{p})\otimes Q_{k})\ex
\end{eqnarray} 
which makes the analogue of system~\eqref{eqn:systemkron}.

\begin{remark}
Recall that to establish stability of (continuous-time)
array~\eqref{eqn:system} it sufficed that $Q_{t}$ is SPSD for each
$t$. (See Theorem~\ref{thm:stability}.) That is, boundedness was not
required. Later, when we established synchronization in
Theorem~\ref{thm:suff}, we needed solely that $Q:\Real\to\Q_{n}$ is
bounded. (See Remark~\ref{rem:bnd}.)  The story has to be a little bit
different in discrete-time.  Note that in
\eqref{eqn:systemdt} we stipulated that $Q_{k}$ (the discrete-time 
counterpart of $Q_{t}$) be in ${\overline\Q}_{n}$.  Even to be able to
establish stability, let alone synchronization, we will need
$|Q_{k}|\leq 1$, a more restrictive condition than
boundedness. Clearly, this has to do with the fact that for
discrete-time systems the magnitude of the righthand side is important
for stability; whereas in continuous time, what matters (for
stability) is only the direction of the righthand side.
\end{remark} 

\begin{lemma}
Let interconnection $\Lambda\in\Real^{p\times p}$ be connected and
$r\in\Real^{p}$ satisfy \eqref{eqn:rdt}. Then, there exists symmetric
positive definite matrix $\Omega\in\Real^{p\times p}$ such that
\begin{eqnarray}\label{eqn:lyapdt}
(\Lambda-\one{r^{T}})^{T}\Omega(\Lambda-\one{r^{T}})-\Omega=-I_{p}\,.
\end{eqnarray}
\end{lemma}

\begin{proof}
We first observe that $(\Lambda-\one r^{T})^{k}=\Lambda^{k}-\one
r^{T}$.  Then we can write $\lim_{k\to\infty}\Lambda^{k}-1r^{T}=0$,
which implies that matrix $[\Lambda-\one r^{T}]$ is Schur, i.e. all of
its eigenvalues are strictly within unit circle. Therefore,
discrete-time Lyapunov equation \eqref{eqn:lyapdt} admits a symmetric
positive definite solution $\Omega$.
\end{proof}

\begin{lemma}\label{lem:Vplus}
Let interconnection $\Lambda\in\Real^{p\times p}$ be connected,
$r\in\Real^{p}$ satisfy \eqref{eqn:rdt}, and symmetric positive
definite matrix $\Omega\in\Real^{p\times p}$ satisfy
\eqref{eqn:lyapdt}. Define $V:\Real^{np}\to\Real_{\geq 0}$ as
$V(\ex):=\ex^{T}(\Omega\otimes I_{n})\ex$.  Then, for all
$Q:\Natural\to\overline{\Q}_{n}$ and all $k\in\Natural$, solution of
system~\eqref{eqn:systemkrondt} satisfies
\begin{eqnarray}\label{eqn:projection}
V(\ex(k+1)-\bar{\ex})-V(\ex(k)-\bar{\ex})
\leq -(\ex(k)-\bar{\ex})^{T}(I_{p}\otimes Q_{k}^{2})(\ex(k)-\bar{\ex})
\end{eqnarray}
where $\bar{\ex}:=(\one{r^{T}}\otimes I_{n})\ex(0)$.
\end{lemma}

\begin{proof}
Observe that $(\one{r^{T}}\otimes I_{n})\ex(k+1)=\ex(k)$ whence
$(\one{r^{T}}\otimes I_{n})\ex(k)=\bar\ex$ for all $k\in\Natural$.
Let $\xi:=\ex-\bar\ex$ and $\Lambda_{\circ}:=\Lambda-\one{r^{T}}$. Then we have
$\xi^{+}=(I_{np}+(\Lambda_{\circ}-I_{p})\otimes
Q_{k})\xi$. We can write
\begin{eqnarray*}
V(\xi^{+})-V(\xi)
&=&\xi^{T}\bigg(\big(I_{np}+(\Lambda_{\circ}-I_{p})\otimes Q_{k}\big)^{T}(\Omega\otimes I_{n})
\big(I_{np}+(\Lambda_{\circ}-I_{p})\otimes Q_{k}\big)-\Omega\otimes I_{n}\bigg)\xi\\
&=&\xi^{T}\bigg(\big((\Lambda_{\circ}-I_{p})^{T}\Omega+\Omega(\Lambda_{\circ}-I_{p})\big)
\otimes(Q_{k}-Q_{k}^{2})
+(\Lambda_{\circ}^{T}\Omega\Lambda_{\circ}-\Omega)\otimes Q_{k}^{2}\bigg)\xi\,.
\end{eqnarray*}
Note that $Q_{k}-Q_{k}^{2}\geq 0$ since $|Q_{k}|\leq 1$ and that
$(\Lambda_{\circ}-I_{p})^{T}\Omega+\Omega(\Lambda_{\circ}-I_{p})<0$
since $\Lambda_{\circ}^{T}\Omega\Lambda_{\circ}-\Omega<0$ (this is
almost immediate when we recall that the sublevel sets of quadratic
Lyapunov functions are convex surfaces.) Therefore
\begin{eqnarray*}
V(\xi^{+})-V(\xi)\leq -\xi^{T}(I_{p}\otimes Q_{k}^{2})\xi\,.
\end{eqnarray*}
Hence the result.
\end{proof}

\begin{remark}
Were $Q_{k}$ a projection matrix, then 
inequality \eqref{eqn:projection} could be replaced by the below equality
\begin{eqnarray*}
V(\ex(k+1)-\bar{\ex})-V(\ex(k)-\bar{\ex})
= -(\ex(k)-\bar{\ex})^{T}(I_{p}\otimes Q_{k})(\ex(k)-\bar{\ex})\,.
\end{eqnarray*}   
\end{remark}

\begin{theorem}\label{thm:stabilitydt}
Given interconnection $\Lambda\in\Real^{p\times p}$, there exists
$\alpha>0$ such that for all $Q:\Natural\to\overline\Q_{n}$, solution of
system~\eqref{eqn:systemkrondt} satisfies
\begin{eqnarray*}
|\ex(k)|\leq\alpha|\ex(0)|
\end{eqnarray*}
for all $k\in\Natural$. 
\end{theorem}

\begin{proof}
Demonstration uses Lemma~\ref{lem:Vplus} and flows similar to that of
Theorem~\ref{thm:stability}.
\end{proof}

\begin{fact}\label{fact:three}
For map $Q:\Natural\to\overline{Q}_{n}$, real number $\varepsilon>0$,
and integer $N\geq 1$ we have
\begin{eqnarray*}
\sigma_{\rm min}\left(\sum_{k=0}^{N-1}Q_{k}\right)\geq\varepsilon
\implies \sigma_{\rm min}\left(\sum_{k=0}^{N-1}Q_{k}^{2}\right)
\geq\frac{\varepsilon^{2}}{N^{2}n^2}\,.
\end{eqnarray*}
\end{fact}

\begin{proof}
Let us be given any $v\in\Real^{n}$ with $v^{T}v=1$. Let $\sigma_{\rm
min}\left(\sum_{k=0}^{N-1}Q_{k}\right)\geq\varepsilon>0$. We can write
\begin{eqnarray*}
v^{T}Q_{0}v+v^{T}Q_{1}v+\ldots+v^{T}Q_{N-1}v\geq \varepsilon
\end{eqnarray*}
which implies that there exists $k^{*}\in\{0,\,1,\,\ldots,\,N-1\}$
such that $v^{T}Q_{k^{*}}v\geq\varepsilon/N$. Since $Q_{k^{*}}$ is an
SPSD matrix there exists an orthogonal matrix $R\in\Real^{n\times n}$
and a diagonal matrix $D\in\Real^{n\times n}$ with entries
$d_{i}\in[0,\,1]$ for $i=1,\,2,\,\ldots,\,n$ such that
$Q_{k^{*}}=R^{T}DR$. Also note that $Q^{2}_{k^{*}}=R^{T}D^{2}R$. Let
$[w_{1}\ w_{2}\ \ldots\ w_{n}]^{T}=w:=Rv$. Note that $w^{T}w=1$ for
$R$ is orthogonal. We can therefore write
\begin{eqnarray*}
\sum_{i=1}^{n}d_{i}w_{i}^{2}\geq\frac{\varepsilon}{N}
\end{eqnarray*}
which implies that there exists $i^{*}\in\{0,\,1,\,\ldots,\,n\}$ such
that $d_{i^{*}}w_{i^{*}}^{2}\geq\varepsilon/(Nn)$. Since
$w_{i^{*}}^{2}\leq 1$, we can write
\begin{eqnarray*}
\frac{\varepsilon^{2}}{N^{2}n^{2}}
&\leq& d_{i^{*}}^{2}w_{i^{*}}^{4}\\
&\leq& d_{i^{*}}^{2}w_{i^{*}}^{2}\\
&\leq& w^{T}D^{2}w\\
&=& v^{T}Q^{2}_{k^{*}}v\\
&\leq& v^{T}\left(\sum_{k=0}^{N-1}Q^{2}_{k}\right)v
\end{eqnarray*}
whence the result follows, for $v$ was arbitrary.
\end{proof}

\begin{theorem}\label{thm:maindt}
Let $\varepsilon>0$ be a real number and $N\geq 1$ an integer. Define
\begin{eqnarray}\label{eqn:deltadt}
\delta(\varepsilon,\,N):=\frac{\varepsilon^{4}}{16N^{8}n^{4}}\,.
\end{eqnarray}
Let interconnection $\Lambda\in\Real^{p\times p}$ be connected,
$r\in\Real^{p}$ and symmetric positive definite matrix
$\Omega\in\Real^{p\times p}$, respectively, satisfy \eqref{eqn:rdt}
and \eqref{eqn:lyapdt}.  Define
\begin{eqnarray}\label{eqn:rhodt}
\rho(\Lambda):=\sigma_{\rm max}(\Omega)\max\{1,\,|\Lambda-I_{p}|^{2}\}\,.
\end{eqnarray}
Let $V(\ex):=\ex^{T}(\Omega\otimes I_{n})\ex$ for
$\ex\in\Real^{np}$. Then, for all $Q:\Natural\to\overline{\Q}_{n}$, the
below inequality
\begin{eqnarray}\label{eqn:Qdt}
\sigma_{\rm min}\left(\sum_{k=0}^{N-1}Q_{k}\right)\geq\varepsilon
\end{eqnarray}
implies that solution of
system~\eqref{eqn:systemkrondt} satisfies
\begin{eqnarray*}
V(\ex(N)-\bar{\ex})
\leq \left(1-\frac{\delta(\varepsilon,\,N)}{\rho(\Lambda)}\right)V(\ex(0)-\bar{\ex})
\end{eqnarray*}
where $\bar{\ex}:=(\one{r^{T}}\otimes I_{n})\ex(0)$.
\end{theorem}

\begin{proof}
Given pair $(\varepsilon,\,T)$ let $\omega:=\varepsilon^{2}/(4N^{3}n^{2})$. 
Consider system~\eqref{eqn:systemkrondt}. Let us introduce
\begin{eqnarray*}
\xi(k):=\ex(k)-\bar{\ex}\,.
\end{eqnarray*}
By Lemma~\ref{lem:Vplus} we have
\begin{eqnarray}\label{eqn:Vxidt}
V(\xi(k+1))-V(\xi(k))\leq -\xi^{T}(k)(I_{p}\otimes Q_{k}^{2})\xi(k)\,.
\end{eqnarray}
Also, $\xi$ can be shown to satisfy
\begin{eqnarray}\label{eqn:xiplus}
\xi^{+}=(I_{np}+(\Lambda-I_{p})\otimes Q_{k})\xi\,.
\end{eqnarray} 
Let $Q:\Natural\to\overline{Q}_{n}$ satisfy \eqref{eqn:Qdt}. Then,
regarding the evolution of $\xi(\cdot)$, one of the two following
cases must be.

{\em Case~1:} $|\xi(k)-\xi(0)|\leq \omega|\xi(0)|$ for all
$k\in\{1,\,2,\,\ldots,\,N-1\}$. Let $b(k):=\xi(k)-\xi(0)$ and recall
that $|Q_{k}|\leq 1$. Let ${\bf Q}_{k}:=(I_{p}\otimes Q_{k})$. Note
that then we have $|{\bf Q}_{k}|\leq 1$ as well as, by
Fact~\ref{fact:three},
\begin{eqnarray*}
\sigma_{\rm min}\left(\sum_{k=0}^{N-1}{\bf Q}_{k}^{2}\right)
&=&\sigma_{\rm min}\left(I_{p}\otimes\sum_{k=0}^{N-1}Q_{k}^{2}\right)\\
&\geq&\frac{\varepsilon^{2}}{N^{2}n^{2}}\,.
\end{eqnarray*}
From \eqref{eqn:Vxidt} we can write 
\begin{eqnarray}\label{eqn:case1dt}
V(\xi(N))
&\leq&V(\xi(0))-\sum_{k=0}^{N-1}\xi^{T}(k){\bf Q}_{k}^{2}\xi(k)\nonumber\\
&=&V(\xi(0))-\sum_{k=0}^{N-1}(\xi(0)+b(k))^{T}{\bf Q}_{k}^{2}(\xi(0)+b(k))\nonumber\\
&=&V(\xi(0))-\xi^{T}(0)\left(\sum_{k=0}^{N-1}{\bf Q}_{k}^{2}\right)\xi(0)
-2\sum_{k=0}^{N-1}b^{T}(k){\bf Q}_{k}^{2}\xi(0)-\sum_{k=0}^{N-1}b^{T}(k){\bf Q}_{k}^{2}b(k)\nonumber\\
&\leq& V(\xi(0))-\frac{\varepsilon^{2}}{N^{2}n^{2}}|\xi(0)|^2+2\omega N|\xi(0)|^{2}\nonumber\\
&=&V(\xi(0))-\frac{\varepsilon^{2}}{2N^{2}n^{2}}|\xi(0)|^{2}\nonumber\\
&\leq&\left(1-\frac{\varepsilon^{2}}{2N^{2}n^{2}\sigma_{\rm max}(\Omega)}\right)V(\xi(0))\,.
\end{eqnarray}

{\em Case~2:} $|\xi(\bar{k})-\xi(0)|\geq\omega|\xi(0)|$ for some
$\bar{k}\in\{1,\,2,\,\ldots,\,N-1\}$. We can by
\eqref{eqn:xiplus} write
\begin{eqnarray}\label{eqn:1dt}
\sum_{k=0}^{\bar{k}}|{\bf Q}_{k}\xi(k)|
&=&|\Lambda-I_{p}|^{-1}
\sum_{k=0}^{\bar{k}}|(\Lambda-I_{p})\otimes I_{n}||{\bf Q}_{k}\xi(k)|\nonumber\\
&\geq&|\Lambda-I_{p}|^{-1}
\sum_{k=0}^{\bar{k}}|((\Lambda-I_{p})\otimes Q_{k})\xi(k)|\nonumber\\
&=&|\Lambda-I_{p}|^{-1}
\sum_{k=0}^{\bar{k}}|\xi(k+1)-\xi(k)|\nonumber\\
&\geq&|\Lambda-I_{p}|^{-1}
\left|\sum_{k=0}^{\bar{k}-1}\xi(k+1)-\xi(k)\right|\nonumber\\
&=&|\Lambda-I_{p}|^{-1}|\xi(\bar{k})-\xi(0)|\nonumber\\
&\geq&|\Lambda-I_{p}|^{-1}\omega|\xi(0)|\,.
\end{eqnarray}
Eq.~\eqref{eqn:1dt} implies that there exists $k\in\{0,\,1,\,\ldots,\,\bar{k}\}$ such that
$|{\bf Q}_{k}\xi(k)|\geq|\Lambda-I_{p}|^{-1}\omega|\xi(0)|/N$ which implies
\begin{eqnarray*}
\sum_{k=0}^{N-1}|{\bf Q}_{k}\xi(k)|^{2}\geq|\Lambda-I_{p}|^{-2}\omega^{2}|\xi(0)|^{2}/N^{2}\,.
\end{eqnarray*}
Then, by \eqref{eqn:Vxidt} 
we can write
\begin{eqnarray}\label{eqn:case2dt}
V(\xi(N))
&=&V(\xi(0))-\sum_{k=0}^{N-1}\xi^{T}(k){\bf Q}_{k}^{2}\xi(k)\nonumber\\
&\leq&V(\xi(0))-\frac{\omega^{2}}{|\Lambda-I_{p}|^{2}N^{2}}|\xi(0)|^{2}\nonumber\\
&\leq&\left(1-\frac{\varepsilon^{4}}{16N^{8}n^{4}|\Lambda-I_{p}|^{2}\sigma_{{\rm max}}(\Omega)}
\right)V(\xi(0))
\end{eqnarray}
The result follows by \eqref{eqn:deltadt}, \eqref{eqn:rhodt},
\eqref{eqn:case1dt}, and \eqref{eqn:case2dt}.
\end{proof}

Theorem~\ref{thm:maindt} suggests the following definition.

\begin{definition}\label{def:suffdt}
Function $Q:\Natural\to \Q_{n}$ is said to be 
{\em sufficiently exciting} if there exists a
sequence of pairs of positive real numbers
$(\varepsilon_{i},\,N_{i})_{i=1}^{\infty}$ satisfying
\begin{eqnarray}\label{eqn:eidt}
\sigma_{\rm min}\left(\sum_{k=k_{i}}^{k_{i}+N_{i}-1}
Q_{k}\right)\geq\varepsilon_{i}
\end{eqnarray}
for $k_{i}=\sum_{j=1}^{i-1}N_{j}$ with $k_{1}=0$, and
\begin{eqnarray}\label{eqn:didt}
\lim_{N\to\infty}\sum_{i=1}^{N}\delta(\varepsilon_{i},\,N_{i})=\infty
\end{eqnarray}
where $\delta(\cdot,\,\cdot)$ is as defined in \eqref{eqn:deltadt}.
\end{definition}

The following result is the discrete-time analogue of
Theorem~\ref{thm:suff}. The proof would have been similar to that of
Theorem~\ref{thm:suff}, had it not been absent from the paper.

\begin{theorem}\label{thm:suffdt}
Let $Q:\Natural\to\overline{\Q}_{n}$ be sufficiently exciting.
Then, for all connected interconnection $\Lambda\in\Real^{p\times p}$,
solutions $x_{i}(\cdot)$ of array~\eqref{eqn:systemdt} synchronize to
\begin{eqnarray*}
\bar{x}(k)\equiv(r^{T}\otimes I_{n})\ex(0)
\end{eqnarray*}
where $r\in\Real^{p}$ satisfies \eqref{eqn:rdt}.
\end{theorem}

Notion of persistence of excitation carries readily to discrete
time. See the below definition.

\begin{definition}
Map $Q:\Natural\to\Q_{n}$ is said to be {\em persistently
exciting} if there exists a pair $(\varepsilon,\,N)$, $\varepsilon>0$
and $N\in\Natural_{\geq 1}$, such that
\begin{eqnarray}\label{eqn:uniformdt}
\sigma_{\rm min}\left(\sum_{k=k_{0}}^{k_{0}+N-1}Q_{k}\right)\geq\varepsilon
\end{eqnarray}
for all $k_{0}\in\Natural$.
\end{definition}

The following theorem is the discrete-time analogue of
Theorem~\ref{thm:suffexp}.  We omit the proof.

\begin{theorem}\label{thm:suffexpdt}
Let interconnection $\Lambda\in\Real^{p\times p}$ be connected and 
function $Q:\Natural\to \overline{\Q}_{n}$ persistently exciting. Then
solutions $x_{i}(\cdot)$ of array~\eqref{eqn:systemdt} exponentially
synchronize to
\begin{eqnarray*}
\bar{x}(k)\equiv(r^{T}\otimes I_{n})\ex(0)
\end{eqnarray*}
where $r\in\Real^{p}$ satisfies \eqref{eqn:rdt}.
\end{theorem}

\subsection{Negative results in discrete time}

Negative results generated in Subsection~\ref{subsec:negative} are not
peculiar to continuous-time arrays. Counterexamples similar to the
ones constructed in the proofs of Theorem~\ref{thm:neg1} and
Theorem~\ref{thm:neg2} can be obtained in discrete time.  We thus
have the following two theorems.

\begin{theorem}
There exist maps $\Lambda:\Natural\to\Real^{p\times p}$, where
$\Lambda_{k}$ is an interconnection for each $k$, and $Q:\Natural\to\overline{\Q}_{n}$ 
such that system $\ex^{+}=(I_{np}+(\Lambda_{k}-I_{p})\otimes
Q_{k})\ex$ has an {\em unbounded} solution.
\end{theorem}

\begin{theorem}
There exist connected interconnection $\Lambda\in\Real^{p\times p}$ and
and map $Q:\Natural\to \overline{\Q}_{n}$ satisfying
\begin{eqnarray*}
\liminf_{N\to\infty}\frac{1}{N}\ \sigma_{\rm min}\left(\sum_{k=0}^{N}Q_{k}\right)>0
\end{eqnarray*}
such that solutions $x_{i}(\cdot)$ of array~\eqref{eqn:systemdt} do
{\em not} synchronize.
\end{theorem}


\subsection{Observability grammian and synchronizability in discrete time} 
As is the case with continuous-time arrays, there is a close relation
between the observability grammian and synchronizability in discrete
time. In this subsection we will provide definitions and theorems
through which we formalize that relation.

\begin{definition}\label{def:asyobsdt}
For $A:\Natural\to\Real^{n\times n}$ and $C:\Natural\to\Real^{m\times n}$,
pair $(C,\,A)$ is said to be {\em asymptotically observable} if 
the summand of 
the observability grammian, $k\mapsto \Phi_{A}^{T}(k,\,0)C^{T}(k)C(k)
\Phi_{A}(k,\,0)$, is sufficiently exciting.
\end{definition}

\begin{definition}\label{def:uniobsdt}
For $A:\Natural\to\Real^{n\times n}$ and $C:\Natural\to\Real^{m\times
n}$, pair $(C,\,A)$ is said to be {\em uniformly observable} if there
exists a pair $(\varepsilon,\,N)$, $\varepsilon>0$ and
$N\in\Natural_{\geq 1}$, such that
\begin{eqnarray*}
\sigma_{\rm min}(W_{\rm o}(k,\,k+N))\geq\varepsilon
\end{eqnarray*}
for all $k\in\Natural$.
\end{definition}

The below result follows from Theorem~\ref{thm:suffdt}.

\begin{theorem}\label{thm:ACmaindt}
Let $A:\Natural\to\Real^{n\times n}$ and $C:\Natural\to\Real^{m\times
n}$ satisfy Assumption~\ref{assume:ACdt} (with constants $\bar{a}$ and
$\bar{c}$.) If pair $(C,\,A)$ is asymptotically observable, then it is
synchronizable. In particular, if we choose
$L:\Natural\to\Real^{n\times n}$ as
\begin{eqnarray}\label{eqn:Ldt}
L(k):=(\bar{a}\bar{c})^{-1}\Phi_{A}(k+1,\,0)\Phi_{A}^{T}(k,\,0)C^{T}(k)
\end{eqnarray} 
then for each $\Lambda\in\G_{>0}$ solutions $x_{i}(\cdot)$ of 
array~\eqref{eqn:arraydt} synchronize to 
\begin{eqnarray*}
\bar{x}(k):=(r^{T}\otimes\Phi_{A}(k,\,0))
\left[\!\!
\begin{array}{c}
x_{1}(0)\\
\vdots\\
x_{p}(0)
\end{array}\!\!
\right]
\end{eqnarray*}
where $r\in\Real^{p}$ satisfy \eqref{eqn:rdt}.
\end{theorem}

Theorem~\ref{thm:suffexpdt} yields the following result.

\begin{theorem}
Let $A:\Natural\to\Real^{n\times n}$ and $C:\Natural\to\Real^{m\times n}$
satisfy Assumption~\ref{assume:ACdt} (with constants $\bar{a}$ and
$\bar{c}$.) Then pair $(C,\,A)$ is synchronizable if
it is uniformly observable. In particular, if we choose $L:\Natural\to\Real^{n\times n}$
as in \eqref{eqn:Ldt} then for each $\Lambda\in\G_{>0}$ solutions $x_{i}(\cdot)$ of 
array~\eqref{eqn:arraydt} exponentially synchronize to 
\begin{eqnarray*}
\bar{x}(k):=(r^{T}\otimes\Phi_{A}(k,\,0))
\left[\!\!
\begin{array}{c}
x_{1}(0)\\
\vdots\\
x_{p}(0)
\end{array}\!\!
\right]
\end{eqnarray*}
where $r\in\Real^{p}$ satisfy \eqref{eqn:rdt}.
\end{theorem}

\section{Conclusion}
We studied synchronization of stable, linear time-varying systems that
are coupled via their outputs. We provided sufficient conditions on
observability grammian for the existence of a bounded linear feedback
law under which the systems synchronize for all fixed connected
interconnections. Related to the main problem, we also studied an
array of coupled integrators with identical time-varying output
matrices that are symmetric positive semi-definite. We showed, via
Lyapunov arguments that, the trajectories of this array stay bounded.
Moreover, if the interconnection is connected and output matrix
satisfies some observability condition, then the systems were
shown to reach consensus.

\bibliographystyle{plain}         
\bibliography{references}            
\end{document}